\documentclass[11pt]{article}

\usepackage[margin=1in]{geometry}
\usepackage{amsmath,amssymb,amsthm,mathtools,bm}
\usepackage{bm}
\usepackage{enumitem}
\usepackage{lmodern}
\usepackage{microtype}
\usepackage[colorlinks=true,linkcolor=blue,citecolor=blue,urlcolor=blue]{hyperref}

\newtheorem{theorem}{Theorem}
\newtheorem{proposition}{Proposition}
\newtheorem{remark}{Remark}
\newtheorem{definition}{Definition}
\newtheorem{lemma}{Lemma}
\newtheorem{lemma*}{Lemma}

\usepackage{makecell}
 \usepackage{float}
\newtheorem{assumption}{Assumption}

\newcommand{\R}{\mathbb{R}}
\newcommand{\Om}{\Omega}
\newcommand{\T}{\mathcal{T}_h}
\newcommand{\dt}{\Delta t}
\newcommand{\relu}{\operatorname{ReLU}}
\newcommand{\diag}{\operatorname{diag}}
\newcommand{\Span}{\operatorname{span}}
\newcommand{\clip}{\operatorname{clip}}
\newcommand{\Tr}{\operatorname{tr}}
\newcommand{\NN}{\mathcal{N}}
\newcommand{\Enc}{\mathcal{E}_h}
\newcommand{\calA}{\mathcal{A}}
\newcommand{\calG}{\mathcal{G}}
\newcommand{\calR}{\mathcal{R}}
\newcommand{\norm}[1]{\left\lVert #1 \right\rVert}
\newcommand{\abs}[1]{\left\lvert #1 \right\rvert}

\title{Numerical Neural Operator: Connection Between Numerical Analysis and the Operator Learning and Building Discretized Neural Operators from Finite Element Methods}
\date{}
\author{
Zecheng Zhang \thanks{Department of Applied and Computational Mathematics and Statistics, University of Notre Dame, Notre Dame IN 46556.  \texttt{Email: zecheng.zhang.math@gmail.com}. Acknowledgment: the research is supported by DOE DE-SC0025440 and NSF NO. 2615564.}
}

\begin{document}

\maketitle
\begin{abstract}
Operator learning aims to learn mappings between function spaces and has been widely used for PDE-related problems.
This raises two fundamental questions. 
First, is there a  connection between standard numerical methods and neural operators, given that both are designed to approximate PDE solution operators? 
Second, when a neural operator is trained on numerically generated data, does it approximate the underlying infinite dimension operator, the training data themselves, or the discretized numerical scheme used to generate them? 
A key hypothesis in the paper is that the trained neural operator may mimic the discretized numerical schemes that generate the training data.
The hypothesis has been partially and preliminary verified by the linear approximation principle proposed in Chen Chen 1995 paper and the following scaling law papers.
However, neural operator approximation from the perspective of numerical discretization remains unexplored.
In this work, we consider a class of PDEs and derive neural operator architectures from the numerical schemes. 
Firstly, the numerical basis functions, such as finite element basis functions, can be approximated or directly constructed by learnable neural networks. 
Then the finite difference schemes for temporal discretization can be unrolled into iterative neural architectures that generate the coefficients of the basis of the resulting neural operator. 
By leveraging numerical analysis, we establish convergence error estimates for these numerically derived neural operators. 
Notably, the network size depends polynomial on the approximation error in $L_2$, improving the nested exponential dependence derived in the general neural operator approximation and aligning with the numerical observations.
Finally, numerical experiments demonstrate reduced errors and improved parameter efficiency.
\end{abstract}

\section{Introduction}
Operator learning aims to learn mappings between function spaces using neural networks, which are often referred to as deep neural operators \cite{chen1993approximations,chen1995universal, back2002universal, lu2022comprehensive,lin2023b, lin2023learning, li2020fourier, li2023fourier, wen2022u, o2022derivate, hao2025laplacian, leung2022nh, zhu2026unsupervised, goswami2023physics, moya2025conformalized, meuris2023machine, cao2025spectral}. 
It has broad applications in PDE-related problems and scientific computing \cite{lu2021learning, li2020fourier,howard2023multifidelity,lu2022multifidelity, sun2020neupde,schaeffer2020extracting2,zhang2019convergence, schaeffer2020extracting}. 
For example, for a time-dependent PDE, one can construct a neural operator to approximate the solution map from the initial condition to the corresponding PDE solution. 
When the initial condition is parametrized, training data can be generated by sampling the parameters of the initial conditions and solving the PDE to obtain pairs of initial conditions and their corresponding solutions. 
After training, given a new initial condition from the same parametrized family, the neural operator can directly predict the corresponding solution without repeatedly solving the underlying PDE.
Another important application is multi-fidelity learning or model closure \cite{zhang2024bayesian, howard2023multifidelity, lu2022multifidelity} (downscaling) or model reduction (upscaling), where the goal is to learn mappings between solutions at different fidelities or scales. 
In particular, a neural operator can take as input a coarse-scale solution, which is relatively inexpensive to obtain, together with limited fine-scale observations at selected locations in the domain. Once trained, the neural operator can reconstruct or predict the full fine-scale solution throughout the domain.
Recently, operator-learning principles have also been successfully applied to large-scale scientific computing problems, making neural operators among the few learning-based approaches that have demonstrated strong potential for real-world scientific applications.

Besides the success of operator learning in applications, the theoretical study of operator learning has also made substantial progress. 
Operator learning was in fact first studied in 1990s' from an approximation-theoretic perspective in \cite{chen1993approximations, chen1995universal}, where the authors showed that neural networks can approximate continuous operators for arbitrary input functions. 
In \cite{lanthaler2022error, deng2022approximation}, convergence rates were analyzed for specific PDE solution operators, while \cite{lanthaler2023curse} later established lower bounds for a broad class of operators.
The Fourier neural operator (FNO), a widely used neural operator architecture \cite{li2020fourier, li2022fourier}, has also been studied in terms of universal approximation and error estimates for specific PDEs \cite{kovachki2021universal, de2022cost}, as well as discretization effects \cite{lanthaler2024discretization}. Another widely used class of DNOs employs encoders and decoders \cite{liu2024deepauto, lanthaler2023operator, bhattacharya2021model} to project the input and output functions into latent spaces \cite{liu2024deep}. 
Such approaches naturally rely on an underlying low-dimensional or subspace structure, and some components of these models may be prescribed rather than learned, making a fully end-to-end analysis difficult. 
In \cite{kovachki2023neural}, the authors unified several neural-operator frameworks, including graph neural operators and FNOs, and established a universal approximation result without explicitly deriving convergence rates. 

Another important work from 1990s is \cite{back2002universal}, where the authors generalize the operator learning setting: they use a function to index an operator in a family of operators, building the fundamental analysis work for the modern multi-operator learning \cite{yang2023context, sun2024towards, zhang2024modno}.
More recently, \cite{liu2024neural} established the first end-to-end convergence-rate analysis for general operators that is not tied to a specific PDE, with rates that align with the corresponding lower-bound estimates.
\cite{bayraktar2025deep} later established the error analysis for operator in probabilistic point of view on unbounded domain using probability language.

Despite rapid advances in operator learning, its relationship with the numerical methods that generate most scientific training data remains poorly understood. 
This leads to a basic but largely unanswered question: is a neural operator learning the underlying continuous PDE operator, or the numerical algorithm used to approximate it?
In practical PDE OL problems (particularly the forward problems), the training data are almost always produced by a numerical solver. 
Hence, the network only observes the discrete input-output map generated by that solver. 
From this viewpoint, a trained neural operator is more naturally interpreted as an approximation of a numerical solution operator (numerical schemes) rather than of the exact infinite-dimensional solution operator. 
This perspective provides a possible explanation for the discrepancy between theoretically predicted approximation rates and the convergence behavior often observed computationally, besides the important explanation from the low-dimensional-structure point of view studied in \cite{liu2024neural, liu2024deep,liu2021besov,yang2025deeponet, liu2025generalization, liu2022benefits}.

Once this viewpoint is adopted, the following question emerges: can a neural operator itself be derived from a classical numerical method? 
Such a connection would not only clarify what neural operators learn, but also connect recent operator-learning architectures with the well-developed mathematical structure of numerical PDEs.
Some initial connections have appeared in \cite{zhang2025finite, meuris2023machine}, where neural networks are introduced to construct learnable finite element basis functions and thereby obtain adaptive local representations. 
Nevertheless, the full connection is still missing: classical FEM determines its coefficients through the solution of a discrete algebraic system, whereas neural operators generate their coefficients through trainable networks. 
Moreover, how spatial and temporal discretizations together give rise to the structure of modern neural operators remains largely unexplored.

Now based on the hypothesis that the trained neural operator may mimic the discretized numerical schemes which generate the training data,  the central idea of this work is to reverse the usual perspective: rather than starting from a neural architecture and asking how well it approximates a PDE operator, we start from the numerical discretization and derive the neural operator from it. 
We show that the resulting architecture naturally exhibits the similar mathematical structure as widely used modern neural operators. 
This establishes a direct bridge between classical numerical analysis and operator learning, and suggests a systematic route for numerical-method-driven neural architecture design. 
We further develop convergence theory for the resulting model, quantifying the $L_2$ approximation error in terms of the sizes of its different neural-operator components. The resulting rates recover convergence behavior commonly seen in numerical computations.

The rest of the paper is organized as follows.
In Section \ref{sec_preliminary}, we introduce the background.
In Section \ref{sec_methods}, we introduce the methods, where we will derive a neural operator according to a numerical discretization.
Later in Section \ref{sec_kpp}, we derive the analysis for a specific PDE.
Notably, the analysis methods could be extended to other types of PDEs.
The numerical examples are shown in Section \ref{sec_numerical}.
Lastly, we present the proof of some numerical results in the Appendix \ref{sec_appendix}.

\section{Preliminary}
\label{sec_preliminary}
Throughout, we consider standard fully connected neural networks with ReLU activation \cite{yarotsky2017error}. 
We define the depth of a network as the number of affine layers and the width as the maximum size of its learnable affine layers. 
The size of a network is the number of its nonzero
affine weights and biases.
If a network is constructed by parallel concatenation of $N_p$ subnetworks, its width is taken to be $N_p$ times the width of the widest subnetwork. We use the notation $\sim$ to denote quantities of the same simtotic order.

The ReLU function can create useful operators. 
We will recall clipping.
Clipping can control the function value to be a range.
For a clipping level $M_0>0$, define $\chi_{M_0}[s]:=\clip_{M_0}[s]:=\min\{M_0,\max\{-M_0,s\}\}.$
The clipping function is represented exactly by a small constant size ReLU network:
$\chi_{M_0}[s]= -M_0+\relu(s+M_0)-\relu(s-M_0).$
\textcolor{black}{Moreover, the clipped square map is globally Lipschitz:}
\begin{equation}\label{eqn_clipped_square_lipschitztz}
 \abs{\chi_{M_0}[a]^2-\chi_{M_0}[b]^2}
 \le
 2M_0\abs{a-b}
 \qquad\forall a,b\in\R.
\end{equation}
The following classical results are from Yarotsky.
The approximation of the ReLU network has been studied by Yarotsky in \cite{yarotsky2017error}.
We recall the following two results.

\begin{lemma}[ReLU approximation to multiplication \cite{yarotsky2017error}]\label{lemma_yarotsky_product}
Let $M_0>0$.  
For every $0<\delta<M_0^2$, there exists a scalar ReLU network $R_\delta:\mathbb{R}^2\to\mathbb{R}$ with constant width, and depth $ D_{\rm sq}(\delta)
 \le
 c_0+c_1\log\left(\frac{M_0^2}{\delta}\right)$
such that
\begin{align*}
    \sup_{s, t\in[-M_0, M_0]}\abs{R_\delta(s, t) - st}\le\delta.
\end{align*}
The network may additionally be chosen so that
$R_\delta(s,0)=R_\delta(0,s)=0$
for every $s\in\mathbb R$.
When $s = t$, the network will give us a network which approximates the square.
\end{lemma}

The next lemma is proved by Yarotsky and relies on Lemma \ref{lemma_yarotsky_product}.
\begin{lemma}[ReLU approximation to Sobolev functions \cite{yarotsky2017error}]
\label{lemma_yarotsky}
Let
$F_{n,d}:=\{f\in W^{n,\infty}([0,1]^d):
\|f\|_{W^{n,\infty}([0,1]^d)}\le1\}$.
For integers $d,n\ge1$ and every $\varepsilon\in(0,1)$,
there is a ReLU architecture such that, for every
$f\in F_{n,d}$, its weights can be chosen to produce
a network $\Phi_{f,\varepsilon}$ satisfying
$\|f-\Phi_{f,\varepsilon}\|_{L^\infty([0,1]^d)}
\le\varepsilon$.
The architecture can be chosen with depth at most
$C(d,n)(1+\log(1/\varepsilon))$, width at most
$C(d,n)\varepsilon^{-d/n}$, and size at most
$C(d,n)\varepsilon^{-d/n}(1+\log(1/\varepsilon))$.
\end{lemma}

\subsection{Neural operator}
\label{sec_methods}

Let ${U}$ and ${V}$ denote the input and output function spaces defined on domains $\Omega_u\subset \mathbb{R}^{d_1}$ and $\Omega_v\subset \mathbb{R}^{d_2}$, respectively, and let $G:{U}\rightarrow{V}$ be an operator. 
The goal of operator learning is to construct a neural operator $G_{\theta}$, parametrized by $\theta$, such that
$|G(u)(x)-G_{\theta}(\hat{u})(x)|$
is small for all $u\in{U}$ and $x\in\Omega_v$, here $\hat{u}$ is a discretization of $u$.

To the best of our knowledge, Chen and Chen \cite{chen1995universal, chen1993approximations, back2002universal} first proposed a neural-network structure for approximating nonlinear operators. 
Specifically, for a continuous operator $G$, under suitable assumptions on $U$ and $V$, they showed that for any $\varepsilon>0$, there exist neural networks $p_j$ and $b_j$ such that
\begin{align}
    G_{\theta}(u)(x)
    = \sum_{j=1}^{N} p_j(\hat{u})\, b_j(x),
    \label{eqn_cc_neural_operator}
\end{align}
approximates $G(u)(x)$ uniformly for admissible $u\in{U}$ and $x\in\Omega_v$, here $\hat{u}$ is a discretization of $u$.
In \cite{liu2024neural}, the authors established the first quantitative convergence-rate analysis for this general operator-learning framework and showed that, to achieve an accuracy $0<\varepsilon<\frac{1}{2}$, the required network size scales as $\varepsilon^{-\varepsilon^{-C}}$, where $C$ depends on $d_1$, $N\sim \varepsilon^{-d_2}$, and the size of the discretization $N_1 \sim \varepsilon^{-d_1}$.
More recently, \cite{yang2026generalization} further analyzed the convergence rate by incorporating the smoothness of the input and output functions.

These theoretical rates are substantially slower than the power-law scaling behavior commonly observed in numerical experiments \cite{lu2021learning}.
One possible explanation is that most numerical studies consider specific PDE solution operators, which may possess additional structure not assumed in general operator approximation theory. 
Moreover, the training data are often generated from low-dimensional parametrizations \cite{liu2024deep}.
For example, input functions are frequently sampled from a truncated Fourier expansion by varying only a finite number of coefficients. 
Consequently, such numerical experiments already impose a low-dimensional structure on the input function space, which can lead to significantly faster observed convergence rates.

\section{Methodology}
We introduce the methodology in this section.
We will study the semilinear parabolic problem and particularly considering the forward Euler scheme in the temporal discretization.
We then use the proposed discretized scheme in weak form to derive a neural operator. 
Note that, different temporal discretization schemes also works.
We choose the forward Euler for the demonstration purpose;
besides, this conditional stable scheme will connect mesh size $h$ and $\Delta t$ through CFL so that we could derive a better analysis results that connect all parts of the error in the next section.
However, to derive neural operator structures, other temporal finite difference schemes also work in similar ways.

\label{sec_fem_general}
\subsection{General semilinear parabolic problem and weak form}
Let $\Omega = [0,1]^d\subset\R^d$ be a bounded Lipschitz domain and let $T>0$.  Consider
\begin{align*}
    \begin{cases}
 v_t-\Delta v=g(v)+f(x,t), & \text{in } \Om\times(0,T),\\
 v=0, & \text{on } \partial\Om\times(0,T),\\
 v(\cdot,0)=v_0, & \text{in } \Om.
\end{cases}
\end{align*}
We assume that $f\in L^2(0,T;C(\overline{\Omega}))$
and that $g:\mathbb R\to\mathbb R$ is locally Lipschitz.
Let $ V:=H_0^1(\Om).$
Multiplying the equation by a test function $w\in V$, integrating over $\Om$, and integrating the Laplacian term by parts gives
\begin{equation}\label{eq:general-weak}
 (v_t,w)+(\nabla v,\nabla w)=(g(v),w)+(f,w)
\end{equation}
where $(a,b):=\int_\Om a(x)b(x)\,dx.$

\subsection{Space and temporal discretizations.}
Let $\{\mathcal T_h\}_h$ be a family of conforming finite element
meshes of $\Omega$.
\textcolor{black}{For an element $K$, let $ h_K:=\operatorname{diam}(K),
 h:=\max_{K\in\T}h_K$.}
For vertex $x_i$ in the mesh, we define $m_i = |\{K\in \mathcal{T}_h: x_i\in K\}|$, then we define  $k_h = \max_{x_i} m_i.$
For simplicial elements, let $\rho_K$ denote the radius of the largest ball contained in $K$.  
We assume the mesh family is regular and quasi-uniformity, i.e., $ \frac{h_K}{\rho_K}\le C_{\rm shape}$
for all elements and all mesh sizes, with $C_{\rm shape}$ independent of $h$, and there exists $c_{\rm qu}>0$, independent of $h$, such that $ c_{\rm qu}h\le h_K\leq h$ for all $K\in\T.$
\medskip

Let $ V_h=\Span\{\phi_1,\dots,\phi_N\}\subset H_0^1(\Om)$
be \textcolor{black}{a conforming finite element space.}  
Write the semidiscrete solution as
\begin{equation}\label{eqn_general_fem_expansion}
 v_h(x,t)=\sum_{j=1}^N (u_h(t))_j\phi_j(x),
\end{equation}
where $u_h(t)\in\R^N$ is the coefficient vector, where we denote $j-$th entry as $(u_h(t))_j = u_j$ for simplicity.
Define the mass matrix $M\in\R^{N\times N}$, stiffness matrix $S\in\R^{N\times N}$,
\textcolor{black}{nonlinear vector $r(u_h)\in\R^N$, and source vector $F(t)\in\R^N$ by}
$ M_{ij}:=\int_\Om \phi_j\phi_i\,dx,
 S_{ij}:=\int_\Om \nabla\phi_j\cdot\nabla\phi_i\,dx,$
 $ r_i(u_h):=\int_\Om g\left(\sum_{j=1}^N u_j\phi_j\right)\phi_i\,dx$, and
 $F_i(t):=\int_\Om f(x,t)\phi_i(x)\,dx.$
Substituting \eqref{eqn_general_fem_expansion} into \eqref{eq:general-weak} and testing with $w_h=\phi_i$ gives
\begin{equation}\label{eq:general-semi-discrete}
 M u_h'(t)+S u_h(t)=r(u_h(t))+F(t).
\end{equation}

There are different choices for the time discretization, throughout the manuscript, we will use the conditional stable forward Euler scheme because the restriction on time discretization and space discretization of the scheme could connect to derive the convergence estimate of the network size. 
However, for the neural operator design, any scheme should work under the similar derivation below.

Let $t_n=n\dt, n=0,1,\dots,N_t, T=N_t\dt.$
Let \textcolor{black}{$u_h^n\in\R^N$} denote the fully discrete coefficient vector at time $t_n$;
besides, define its finite element reconstruction
by $v_h^n(x):=\sum_{j=1}^N(u_h^n)_j\phi_j(x)$.
\textcolor{black}{Since $f\in L^2(0,T;C(\overline{\Omega}))$, we then define}
\begin{equation}\label{eq:F-time-average}
 F_i^n:=\frac1{\dt}\int_{t_n}^{t_{n+1}}\int_\Om f(x,t)\phi_i(x)\,dxdt,
 \quad i=1,\dots,N, \quad F^n = [F_1^n, ..., F_N^n]^{\intercal}.
\end{equation}
The Euler step for \eqref{eq:general-semi-discrete} is $ M\frac{u_h^{n+1}-u_h^n}{\dt}+S u_h^n=r(u_h^n)+F^n.$
Multiplying by $\dt$, moving known terms to the right-hand side, and applying $M^{-1}$ gives
\begin{align*}
     u_h^{n+1}
 =
 \left(I-\dt M^{-1}S\right)u_h^n
 +
 \dt M^{-1}r(u_h^n)
 +
 \dt M^{-1}F^n.
\end{align*}
For the analysis part, we will assume the scheme is stable by enforcing the following CFL conditions. 
Note, for the neural operator structural design, no CFL (it is used only in the analysis) is needed if the scheme is unconditional stable.
\begin{assumption}
\label{assumption_cfl} 
We will assume the following CFL conditions.
\begin{align*}
     \dt\lesssim h^2, \text{ and } N_t = \frac{T}{\Delta t}\gtrsim  h^{-2}.
\end{align*}
To the diffusion update is nonexpansive in the standard mass norm (defined in Section \ref{sec_norms}) $\|(I-\Delta tM^{-1}S)a\|_M\le \|a\|_M, a\in\mathbb R^N$, we will enforce $ \dt\lambda_{\max}(M^{-1}S)\le 2$. 
\textcolor{black}{Using the inverse inequality, it follows that $ \lambda_{\max}(M^{-1}S)\le C h^{-2}.$}
A derivation is introduced in the Appendix \ref{section_cfl_appendix} for the completeness of the manuscript.
For the network complexity estimates below, we additionally
choose $\Delta t\sim h^2$ subject to
Assumption~\ref{assumption_cfl}, so that $N_t\sim h^{-2}$.
\end{assumption}

\subsection{Quadrature representation for the nonlinear terms and FEM solution.}
\label{sec_quadrature}
We use quadrature to evaluate the nonlinear integrals.
Let the global quadrature points and weights be
$\{x_q,\omega_q\}_{q=1}^Q$.
The algebraic representation below applies to any quadrature
rule for which the required point evaluations are defined.
For the analysis, we assume $\omega_q>0$, together with the
additional quadrature conditions stated in
Assumption~\ref{assumption_quadrature}.
Positive weights are not automatic for every high-order
quadrature rule.
The following remark gives an example.

\begin{remark}
\label{remark_quadrature_example}
A standard example with positive weights is the
Gauss--Legendre rule.
On the reference interval $[-1,1]$, the $n_G$-point
Gauss--Legendre rule has nodes $\xi_1,\ldots,\xi_{n_G}$,
the roots of the Legendre polynomial $P_{n_G}$, and weights
$\widehat\omega_i
=
\frac{2}{(1-\xi_i^2)[P_{n_G}'(\xi_i)]^2},
\quad i=1,\ldots,n_G.$
Hence $\widehat\omega_i>0$.
The rule satisfies
\begin{align*}
\int_{-1}^1 p(\xi)\,d\xi
=
\sum_{i=1}^{n_G}\widehat\omega_i p(\xi_i),
\quad
\forall p\in\mathbb P_{2n_G-1}([-1,1]).
\end{align*}
On a physical interval $[a,b]$, the affine map
$x=(a+b)/2+(b-a)\xi/2$ gives nodes $x_i$ and positive
weights $\omega_i=(b-a)\widehat\omega_i/2$.
Therefore, to integrate all one-dimensional polynomials
of degree at most $3r$ exactly, it is enough to choose
$n_G\ge\left\lceil(3r+1)/2\right\rceil$.
The degree-$3r$ requirement is used in the Fisher--KPP
analysis below, since $v_h^2\phi_i$ has degree at most
$3r$ on each affine simplicial element for $P_r$
finite elements.
\medskip

On affine tensor-product quadrilateral or hexahedral
elements, one takes the tensor product of this
one-dimensional rule.
The tensor-product weights remain positive, and the rule
is exact for polynomials whose degree in each reference
coordinate is at most $2n_G-1$.
With the above choice of $n_G$, this includes polynomials
of degree at most $3r$ in each reference coordinate,
as required for products of three $Q_r$ finite element
functions.
On simplex elements, one may instead choose any positive
elementwise quadrature rule exact on $\mathbb P_{3r}$,
when such a rule is available for the desired degree
and element type.
Positivity and degree-$3r$ exactness are stated explicitly
in Assumption~\ref{assumption_quadrature}.
\end{remark}

We now define the evaluation matrix $B\in\mathbb R^{Q\times N}$
and the quadrature weight matrix $W\in\mathbb R^{Q\times Q}$ by
\begin{align}\label{eq_general_B_W}
B_{qj}:=\phi_j(x_q),
\quad
W:=\operatorname{diag}(\omega_1,\dots,\omega_Q).
\end{align}
For any vector $u\in\mathbb R^N$,
$(Bu)_q=\sum_{j=1}^N u_j\phi_j(x_q)$.
Thus, if $v_h=\sum_j u_j\phi_j$, then
$(Bu)_q=v_h(x_q)$.

To remain consistent with \eqref{eq:F-time-average},
define the time-averaged source by
$\bar f^n(x):=\Delta t^{-1}\int_{t_n}^{t_{n+1}}f(x,t)\,dt$,
for $n=0,\ldots,N_t-1$.
Under the assumption
$f\in L^2(0,T;C(\overline{\Omega}))$,
these averages have well-defined spatial point values.
Let $f_Q^n\in\mathbb R^Q$ denote the corresponding sampled
source vector, with $(f_Q^n)_q:=\bar f^n(x_q)$.
\medskip

We now derive the matrix quadrature formula.
For $v_h=\sum_j(u_h)_j\phi_j$, the nonlinear vector satisfies
$r_i(u_h)=\int_\Omega g(v_h(x))\phi_i(x)\,dx$.
Applying the quadrature rule gives
$r_i(u_h)\approx
\sum_{q=1}^Q\omega_q g(v_h(x_q))\phi_i(x_q)$.
Since $v_h(x_q)=(Bu_h)_q$ and $\phi_i(x_q)=B_{qi}$,
this becomes
$r_i(u_h)\approx
\sum_{q=1}^Q B_{qi}\omega_q g((Bu_h)_q)$.
The right-hand side is precisely the $i$-th component
of $B^TWg(Bu_h)$, where $g$ acts componentwise.
Thus $r(u_h)\approx B^TWg(Bu_h)$.
Similarly,
\begin{align*}
F_i^n
&=\int_\Omega \bar f^n(x)\phi_i(x)\,dx \approx
\sum_{q=1}^Q\omega_q\bar f^n(x_q)\phi_i(x_q), \text{ or } F^n\approx B^TWf_Q^n.
\end{align*}

Denote $A_h:=I-\Delta tM^{-1}S$ and
$C_h:=\Delta tM^{-1}B^TW$.
Replacing the nonlinear and source loads by these quadrature
approximations defines the quadrature-based Euler scheme.
The update is then,
\begin{align}
\label{eqn_one_branch_block}
u_h^{n+1}
=
A_hu_h^n+C_hg(Bu_h^n)+C_hf_Q^n
=:L_n(u_h^n;f_Q^n).
\end{align}
Thus one numerical time step has the structure of a shallow
network block: an affine evaluation map $u\mapsto Bu$,
a scalar nonlinearity applied componentwise, an affine
recombination map $z\mapsto C_hz$, a residual linear branch
$u\mapsto A_hu$, and a source branch $f_Q^n\mapsto C_hf_Q^n$.
This block uses $g$ as its scalar nonlinearity.
In the ReLU construction below, the nonlinear evaluation
is replaced by a suitable ReLU approximation.

By composition at $t=t_m$,
\begin{align}\label{eqn_general_composition}
u_h^m
=
L_{m-1}(\cdot;f_Q^{m-1})
\circ\cdots\circ
L_0(\cdot;f_Q^0)(u_h^0).
\end{align}
The fully discrete FEM solution at $t=t_m$ is then
$v_h^m(x)=\sum_{j=1}^N(u_h^m)_j\phi_j(x)$.

\subsubsection{Network structure}
\label{sec_nn_structure}
Notably, the above representation has the same structural form as the operator approximation in \eqref{eqn_cc_neural_operator}. 
In the language of operator learning, the finite element coefficients $(u_h^m)_j$ play the role of the coefficient maps $p_j$, carrying the dependence on the input function (the initial condition and the source term if the PDE has), while the finite element basis functions $\phi_j$ play the role of the output basis functions $b_j$.

More precisely, if the source term varies across tasks, the operator input is 
$u =(v_0,f)$; if the source term is fixed across the data set, the input reduces to $u = v_0$. 
Using \eqref{eqn_general_composition}, the finite element approximation at time $t_m$ can therefore be written in the operator-learning form
\begin{equation*}
G(v_0,f)(x,t_m)
\approx
\sum_{j=1}^N
\left[
L_{m-1}(\cdot;f_Q^{m-1})\circ\cdots\circ L_0(\cdot;f_Q^0)(u_h^0)
\right]_j
\phi_j(x).
\end{equation*}
Accordingly, the coefficient maps are
\begin{equation*}
p_j(v_0,f)
:=
\left[
L_{m-1}(\cdot;f_Q^{m-1})\circ\cdots\circ L_0(\cdot;f_Q^0)(u_h^0)
\right]_j.
\end{equation*}
Thus, the finite element basis functions $\phi_j$ provide the output basis $b_j$, while the time-unrolled numerical scheme defines the coefficient maps $p_j$.


To construct a neural operator based on FEM, we need neural-network realizations of both $L_n$ and the finite element basis functions $\phi_j$.
First, for $L_n$ in \eqref{eqn_one_branch_block}, the matrices $A_h$, $C_h$, and $B$ are replaced by learnable linear maps, while the nonlinear term $g$ is represented by nonlinear neural-network layers. 
The Fisher--KPP equation provides an example for the nonlinear network approximation for $g$; see Equation \eqref{def_network_predicted_traj} and Lemma \ref{lemma_yarotsky_product} for a mathematically justified construction of a special network block for approximating $g$. 
However, in general implementations, $g$ can be represented by a shallow fully connected neural network. See Figure \ref{fig_one_layer} for the demonstration of the one learnable layer $L_n$.
\begin{figure}[H]
    \centering
    \includegraphics[width=0.8\linewidth]{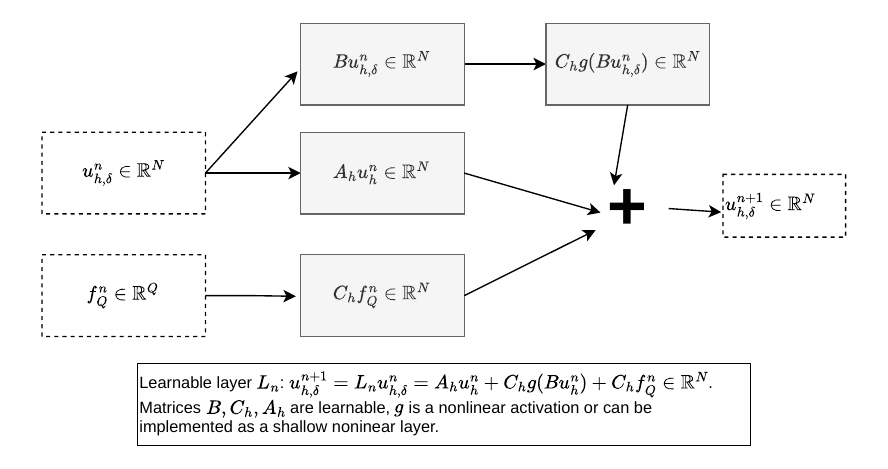}
    \caption{One learnable layer. See figure \ref{fig_full_nn} for the entire network structure.}
    \label{fig_one_layer}
\end{figure}
\begin{figure}[H]
    \centering
    \includegraphics[width=0.8\linewidth]{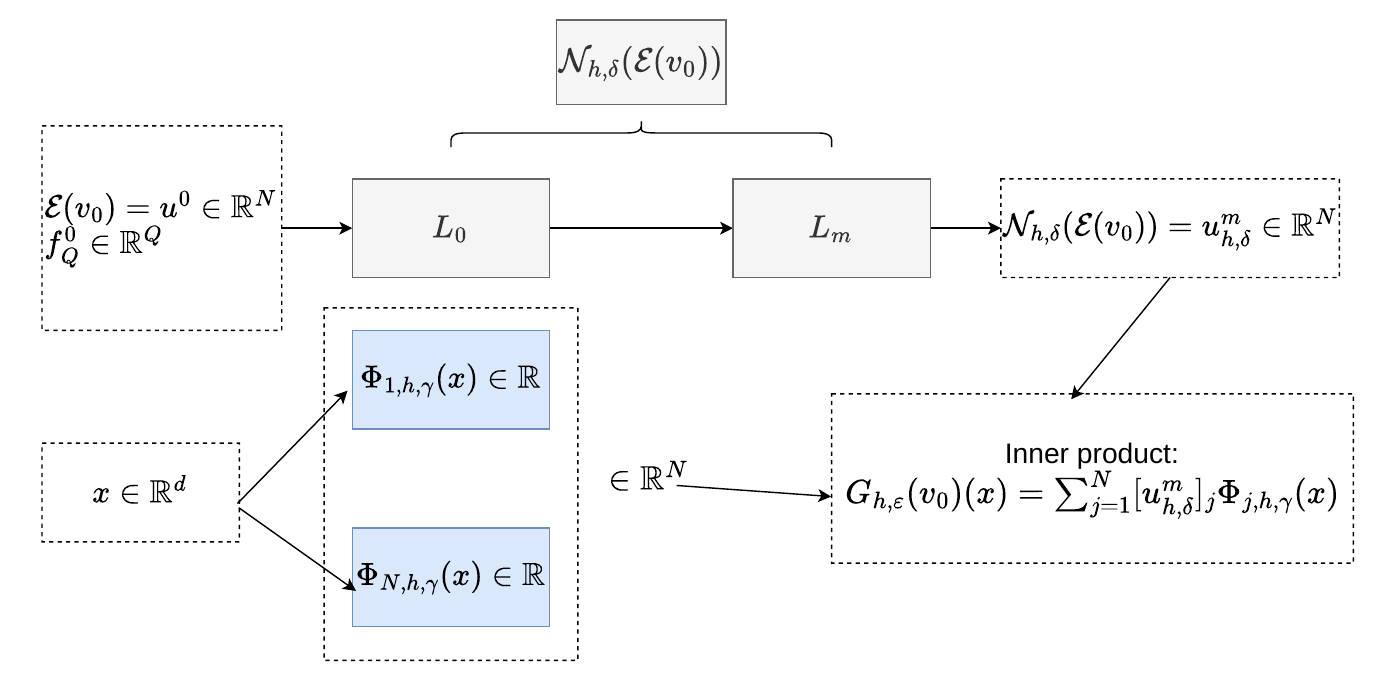}
    \caption{Entire network $\mathcal{G}_{h, \varepsilon}(\cdot)(\cdot)$. 
See figure \ref{fig_one_layer} for the demonstration of one learnable layer $L_n$. 
    Notably, $P_1$ FEM can be constructed by ReLU network exactly, see \cite{he2020relu}. 
    Besides, compared to neural operators supported by the approximation principle \ref{eqn_cc_neural_operator} and theory \cite{chen1995universal,liu2024neural}, the proposed mathematically verified method does not need separate individual neural network to generate each individual coefficient $[u_{h, \delta}^m]_j$, which allows computations to be shared among the coefficient outputs.
    Unstacked DeepONet \cite{lu2021learning} similarly generates
all basis coefficients through a single branch network.
The present work complements this design with explicit bounds on neural scaling results.
    }
    \label{fig_full_nn}
\end{figure}

Second, we need a neural-network realization of the finite element basis functions $\phi_j$. 
For $P_1$ FEM, or equivalently $Q_1$ FEM in one dimension, each basis function $\phi_j$ can be represented exactly by a ReLU neural network \cite{he2020relu}; see Lemma \ref{lemma_jinchao_xu}. 
For higher-order FEM, the basis functions $\phi_j$ can be approximated by ReLU neural networks; see Lemma \ref{lemma_basis_network_nonlocal} and Lemma \ref{lemma_basis_network_local} for two different constructions.
The resulting entire network architecture is illustrated in Figure \ref{fig_full_nn}.

\section{Fisher--KPP Study and Analysis}\label{sec_kpp}
Starting from this section, we will study the error analysis for a nonlinear equation with $ g(v)=v-v^2$ and $f = 0$.
Similar methods can be applied to other analyze other equations with different numerical discretization in time when the numerical scheme error can be established.
We separate out the linear term $v$ and define
\begin{align*}
     r(u_h)=Mu_h-q(u_h),
 \quad
 q_i(u_h):=\int_\Om \left(\sum_{j=1}^N u_j\phi_j\right)^2\phi_i\,dx,
\end{align*}
then $q(u_h) = [q_1(u_h), ..., q_N(u_h)]^{\intercal}$.
The semidiscrete system becomes $M u_h'(t)+S u_h(t)=M u_h(t)-q(u_h(t)),$
or equivalently $M u_h'(t)+(S-M)u_h(t)=-q(u_h(t)).$
At time $t_n=n\dt$, let $u_h^n\in\R^N$ denote the fully discrete Fisher--KPP coefficient vector.  
The forward Euler discretization is
$ M\frac{u_h^{n+1}-u_h^n}{\dt}+S u_h^n=M u_h^n-q(u_h^n).$
Rearranging gives $u_h^{n+1}
 =
 \left(I+\dt I-\dt M^{-1}S\right)u_h^n
 -
 \dt M^{-1}q(u_h^n).$
Define $L_h$, we then have the update,
\begin{equation}\label{eq_kpp_Lh}
 L_h:=I+\dt I-\dt M^{-1}S, \quad
  u_h^{n+1}=L_hu_h^n-\dt M^{-1}q(u_h^n).
\end{equation}

\subsection{Quadrature in the Fisher--KPP update}\label{subsec_kpp-quad}
Denote $M_Q\in\mathbb{R}^{N\times N}$ as the quadrature approximation to the mass matrix. 
Specifically, $ (M_Q)_{ij}:=\sum_{q=1}^Q\omega_q\phi_j(x_q)\phi_i(x_q).$
Using $B_{qj}=\phi_j(x_q)$ and $W=\diag(\omega_1,\dots,\omega_Q)$, we have
\begin{align}\label{eq:MQ}
 (B^TWB)_{ij}
 =
 \sum_{q=1}^Q B_{qi}\omega_qB_{qj}
 =
 \sum_{q=1}^Q\omega_q\phi_i(x_q)\phi_j(x_q)
 =
 (M_Q)_{ij}.
\end{align}
Thus $ M_Q=B^TWB.$
If the quadrature rule is exact for the products $\phi_i\phi_j$, then $M_Q=M = B^TWB$.  
For affine elements with $\mathbb{P}_r$ basis functions, $\phi_i\phi_j$ is a polynomial of degree at most $2r$ on each element; hence elementwise quadrature exactness through degree $2r$ implies $B^TWB=M$.  Below we impose the stronger degree-$3r$ exactness condition because it also makes the Fisher--KPP square term exact.  
\begin{assumption}[Quadrature]
\label{assumption_quadrature}
Let $\widehat K$ be the reference element. 
We fix a quadrature
rule $\{(\widehat x_\ell,\widehat\omega_\ell)\}_{\ell=1}^{n_q}$
that is exact for $\mathbb P_{3r}(\widehat K)$ and has positive
weights $\widehat\omega_\ell>0$. 
The number $n_q$ may depend
on the fixed polynomial degree $r$ and dimension $d$, but
is independent of $h$.
On each element $K\in\mathcal T_h$, we use the affine image
of this reference rule.
\end{assumption}
The resulting quadrature rule on each element $K$ has positive
weights and is exact for $\mathbb P_{3r}(K)$.
Let $\{(x_q,\omega_q)\}_{q=1}^Q$ denote the assembled
elementwise quadrature entries, counting shared locations
separately. Then $Q=n_q\#\mathcal T_h$.
Since the mesh family is shape-regular and quasi-uniform,
$Q\sim h^{-d}$, with constants independent of $h$.

For the nonlinear square term, let $u_h\in\mathbb R^N$
and set $v_h=\sum_{j=1}^N(u_h)_j\phi_j$.
For each element $K$, let $Q_K$ denote the indices of
the quadrature entries assigned to $K$, and define
$(q_Q^K(u_h))_i
:=
\sum_{q\in Q_K}\omega_q v_h(x_q)^2\phi_i(x_q).$
Since $v_h|_K,\phi_i|_K\in\mathbb P_r(K)$, we have
$v_h^2\phi_i\in\mathbb P_{3r}(K)$.
Assumption~\ref{assumption_quadrature} therefore gives
$(q_Q^K(u_h))_i=\int_K v_h(x)^2\phi_i(x)\,dx$.
Summing over all elements yields
\begin{align*}
    (q_Q(u_h))_i
=
\sum_{q=1}^Q\omega_q(Bu_h)_q^2B_{qi}
=
\left[B^TW(Bu_h)^{\odot2}\right]_i
=
q_i(u_h),
\end{align*}
where $(Bu_h)^{\odot2}$ denotes componentwise squaring.
Consequently, the Assumption \ref{assumption_quadrature} implies the following.
\begin{equation}\label{eq:kpp-exact-quadrature-consequences}
 B^TWB=M,
 \quad
 q_Q(u_h)=q(u_h),
 \quad
 \sum_{q=1}^Q\omega_q=|\Omega|.
\end{equation}

\begin{remark}[What changes if $M_Q\ne M$]
The formula $L_h=I+\dt I-\dt M^{-1}S$ contains the linear reaction contribution coming from the $+v$ part of $v-v^2$.  If the quadrature is not mass-exact, then the quadrature approximation of the linear reaction vector is $M_Qu$, not $Mu$.  The quadrature-based KPP update would then be
\begin{align*}
     u_h^{n+1}
 =
 \left(I-\dt M^{-1}S+\dt M^{-1}M_Q\right)u_h^n
 -
 \dt M^{-1}B^TW(Bu_h^n)^{\odot2}.
\end{align*}
Equivalently, the difference $\dt M^{-1}(M_Q-M)u_h^n$ would be an additional quadrature consistency term.  In the rest of this note we do not pursue this case; we assume degree-$3r$ exact quadrature, so $M_Q=M$ and $q_Q=q$.
\end{remark}

Now with a slight abuse of notation, we redefine $C_h$ to absorb
the negative sign of the square reaction term:
\begin{equation}\label{eq_kpp_Ch}
 C_h:=-\dt M^{-1}B^TW\in \mathbb{R}^{N\times Q}
\end{equation}
Thus $C_h$ maps a vector of quadrature-point values in $\R^Q$ to a coefficient vector in $\R^N$.
Then the iterative update \eqref{eq_kpp_Lh} becomes
\begin{equation}\label{eq:kpp-one-step-final}
 u_h^{n+1}=L_hu_h^n+C_h(Bu_h^n)^{\odot2}.
\end{equation}
Let $ \Psi_h(u):=L_hu+C_h(Bu)^{\odot2}.$
Then $    u_h^{N_t}=\underbrace{\Psi_h\circ\cdots\circ\Psi_h}_{N_t\text{ times}}(u_h^0).$
The associated finite element function at time step $n$, namely at physical time $t_n=n\dt$, is $ v_h^n(x):=\sum_{j=1}^N (u_h^n)_j\phi_j(x).$
\textcolor{black}{The notation $u_j^n$, sometimes used informally, means the $j-$th component $(u_h^n)_j$, but we avoid it here to prevent ambiguity.}

\subsection{Preparation for analysis: norms and operator bounds}
\label{sec_norms}
In this section we estimate all norms of the matrices defined before.
For a coefficient vector $e\in\R^N$, define the standard mass-matrix norm
$\norm{e}_M:=\sqrt{e^TMe}.$
Let $\calR_he:=\sum_{j=1}^N e_j\phi_j\in V_h$, with $e$ feasible.
Then the $M$-norm of the coefficient vector is exactly the $L^2(\Omega)$-norm of the corresponding finite element function:
\begin{align}
 \norm{\calR_he}_{L^2(\Om)}^2
 =
 \int_\Omega
 \left(\sum_{j=1}^N e_j\phi_j\right)
 \left(\sum_{i=1}^N e_i\phi_i\right)dx 
 =
 \sum_{i,j=1}^N e_i e_j\int_\Omega \phi_j\phi_i\,dx 
 =
 e^TMe
 =
 \norm{e}_M^2.
 \label{eq:M-L2}
\end{align}

Similarly, we use the quadrature weight matrix $W$ \eqref{eq_general_B_W} to define the $W-$norm.  
Since the quadrature weights are assumed positive in Assumption \ref{assumption_quadrature},
$\norm{z}_W:=\sqrt{z^TWz}=\biggl(\sum_{q=1}^Q\omega_q z_q^2\biggr)^{1/2}$
\textcolor{black}{defines a norm on $\R^Q$. } 
This is the quadrature analogue of an $L^2$-norm for values stored at quadrature points.

We norm define the induced norms for operator $L_h$ in \eqref{eq_kpp_Lh}, $C_h$ in \eqref{eq_kpp_Ch} and $B$ in \eqref{eq_general_B_W} as follows.
\begin{align}\label{eqn_operator_norms_1}
 \norm{L_h}_{M\to M}:=\sup_{e\ne0}\frac{\norm{L_he}_M}{\norm{e}_M},
 \quad
 \norm{C_h}_{W\to M}:=\sup_{z\ne0}\frac{\norm{C_hz}_M}{\norm{z}_W},
 \quad
  \norm{B}_{M\to W}:=\sup_{e\ne0}\frac{\norm{Be}_W}{\norm{e}_M}.
\end{align}
These choices are natural for the two maps that appear in one time step.  
The following lemma gives an estimate on the operator norms of $L_h, C_h$ and $B$.

\begin{lemma}[Basic stability bounds]\label{lemma_numerical_operator_bounds}
Assume the Assumption \ref{assumption_quadrature} hols.
Then we have $ \norm{B}_{M\to W}=1$, $\norm{C_h}_{W\to M}\le \dt$, $ \sqrt{\Tr(W)}=|\Om|^{1/2}.$
Finally, if the explicit diffusion CFL condition $\dt\lambda_{\max}(M^{-1}S)\le2$ holds, then $ \norm{L_h}_{M\to M}\le1+\dt.$
\end{lemma}
The proof uses the standard numerical techniques, we leave the proof to Appendix \ref{sec_stability_bnds}.

\subsection{Predicted coefficient trajectory error}\label{sec_network_error}
In this section, we will first analyze the error in approximating the iterative coefficients.
The key point is that we do not assume a pointwise bound on the network iterates.  
We assume that the reference finite element solution trajectory is bounded at the quadrature points.  
The clipped square map is globally Lipschitz, so the perturbation argument remains valid even if $Bu_{h,\delta}^n$ leaves the interval $[-M_0,M_0]$, where $u_{h,\delta}^n$ is the network approximation to $u_h^n$.

\begin{assumption}[Reference numerical boundedness at quadrature points]\label{assumption_reference_bounded}
There exists $M_*>0$, independent of $h$, $\dt$, $n$, $\delta$, and the initial condition in the admissible class, such that the reference fully discrete solution satisfies
\begin{align*}
     \abs{(Bu_h^n)_q}\le M_*,
 q=1,\dots,Q,
 n=0,\dots,N_t.
\end{align*}
We choose a clipping level $M_0\ge M_*$, where $M_0$ is independent of the $h, \Delta t$ and initial condtion.
\end{assumption}

\begin{remark}[Reasonableness of
Assumption~\ref{assumption_reference_bounded}]
This assumption may follow from a discrete maximum principle
or a uniform $L^\infty$-stability estimate for the reference
numerical solution.
In particular, a method preserving $[0,1]$ at the quadrature
points gives $M_*=1$ for admissible initial data.
Such boundedness is not automatic for consistent-mass
Galerkin schemes under the diffusion CFL condition alone,
so we state it explicitly as an assumption.
\end{remark} 

We now define the network prediction trajectory, and the coresponding solutions.
\begin{definition}
\label{def_network_predicted_traj}
    Let $\tilde{R}_\delta$ be the scalar square network from Lemma~\ref{lemma_yarotsky_product} with $\delta$ accuracy with domain $[-M_0, M_0]$.
Define the network-implemented coefficient $u_{h, \delta}^n\in\mathbb{R}^{N}$ by
\begin{align*}
     u_{h,\delta}^{n+1}=L_hu_{h,\delta}^n+C_h{R}_\delta(Bu_{h,\delta}^n), \text{ where } R_{\delta} = \tilde{R}_{\delta}\circ \chi_{M_0}.
\end{align*}
Here $\chi_{M_0}$ is the clipping operator implemented by ReLU, $\tilde{R}_{\delta}$ is applied to $\chi_{M_0} (Bu_{h, \delta}^n)$ entry-wise, and  initial finite element encoding ($n = 0$) is the same as as the reference scheme $ u_{h,\delta}^0=u_h^0.$
We then call the associated finite element function  $v_{h,\delta}^n(x):=\sum_{j=1}^N (u_{h,\delta}^n)_j\phi_j(x)$ semi-network approximation.
Define the coefficient error $ e^n:=u_h^n-u_{h,\delta}^n.$
\end{definition}

\begin{lemma}\label{lem:recursion}
Let $u_{h, \delta}^n$, $e^n$ and network $R_\delta$ be defined as in Definition \ref{def_network_predicted_traj}.
Assume the boundedness Assumption~\ref{assumption_reference_bounded} and let $M_0\ge M_*$.  Then,
\begin{align*}
     \norm{e^{n+1}}_M
 \le
 \left(\norm{L_h}_{M\to M}+2M_0\norm{C_h}_{W\to M}\norm{B}_{M\to W}\right)\norm{e^n}_M
 +
 \norm{C_h}_{W\to M}\sqrt{\Tr(W)}\,\delta.
\end{align*}
\end{lemma}

\begin{proof}
Because $\abs{(Bu_h^n)_q}\le M_*\le M_0$, clipping does not change the reference solution values $ \chi_{M_0}(Bu_h^n)=Bu_h^n$.  Therefore the reference update \eqref{eq:kpp-one-step-final} can be written as
$u_h^{n+1}=L_hu_h^n+C_h\chi_{M_0}(Bu_h^n)^{\odot2}.$
Subtracting the network approximation to coefficients \eqref{def_network_predicted_traj} from the reference update equation above gives
$e^{n+1}
 =
 L_he^n
 +
 C_h\left(
 \chi_{M_0}(Bu_h^n)^{\odot2}-R_\delta(Bu_{h,\delta}^n)
 \right).$
For the second term of the nonlinear difference, we add and subtract, 
\begin{align*}
\tilde{e}^n &= \chi_{M_0}(Bu_h^n)^{\odot2}-R_\delta(Bu_{h,\delta}^n)\notag\\
 &=
 \left[
 \chi_{M_0}(Bu_h^n)^{\odot2}
 -
 \chi_{M_0}(Bu_{h,\delta}^n)^{\odot2}
 \right]
 +
 \left[
 \chi_{M_0}(Bu_{h,\delta}^n)^{\odot2}
 -
 R_\delta(Bu_{h,\delta}^n)
 \right].
\end{align*}
For the first bracket of the above equation for of $\tilde{e}^n$
and the $q-$th entry, we apply the Lipschitz estimate \eqref{eqn_clipped_square_lipschitztz} of the clipping,
$$
 \abs{\chi_{M_0}(Bu_h^n)_q^2-\chi_{M_0}(Bu_{h,\delta}^n)_q^2}
 \le
 2M_0\abs{(B e^n)_q},  q = 1,2,.., Q.
 $$
We then consider the entire vector's norm estimate.
This estimate is naturally written in the $W$-norm because this bracket is still a vector of nonlinear values at quadrature points, before the matrix $C_h$ maps it back into coefficient space.  
Squaring each entry, multiplying by $\omega_q$, and summing over $q$ yields
\begin{equation}\label{eqn_clipped_lipschitz_W}
 \norm{\chi_{M_0}(Bu_h^n)^{\odot2}-\chi_{M_0}(Bu_{h,\delta}^n)^{\odot2}}_W
 \le
 2M_0\norm{Be^n}_W
 \le
 2M_0\norm{B}_{M\to W}\norm{e^n}_M.
\end{equation}
For the second bracket of $\tilde{e}^n$, by Lemma~\ref{lemma_yarotsky_product}, we have
$\abs{\chi_{M_0}(Bu_{h,\delta}^n)_q^2- R_\delta((Bu_{h,\delta}^n)_q)}\le\delta.$
It follows the estimate for the entire vector in $W-$norm,
\begin{equation}\label{eq:clipped-approx-W}
 \norm{\chi_{M_0}(Bu_{h,\delta}^n)^{\odot2}-R_\delta(Bu_{h,\delta}^n)}_W
 \le
 \sqrt{\Tr(W)}\,\delta.
\end{equation}
Using \eqref{eqn_clipped_lipschitz_W} and \eqref{eq:clipped-approx-W}, we can now estimate $\tilde{e}^n$,
\begin{align*}
     \norm{\tilde{e}^n}_W
 \le
 2M_0\norm{B}_{M\to W}\norm{e^n}_M+\sqrt{\Tr(W)}\,\delta.
\end{align*}
The triangle inequality in the $M$-norm gives the estimate for $e^{n}$,
\begin{align*}
     \norm{e^{n+1}}_M
 \le
 \norm{L_he^n}_M+\norm{C_h \tilde{e}^n}_M\leq  \norm{L_h}_{M\to M}\norm{e^n}_M +  \norm{C_h}_{W\to M}\norm{\tilde{e}^n}_W. 
\end{align*}
Using the definition of the induced operator norms and Lemma \ref{lemma_numerical_operator_bounds} prove the results.
\end{proof}
The next proposition estimates the cumulative error.
\begin{proposition}[Accumulated error]\label{prop_accumulative_error}
Assume Assumption \ref{assumption_quadrature} (quadrature) and the exact-quadrature consequences \eqref{eq:kpp-exact-quadrature-consequences}, Assumption \ref{assumption_cfl} (CFL), Assumption~\ref{assumption_reference_bounded} (numerical bounded).
Let the predicted trajectory $u_{h, \delta}^n$, semi-network approximation $v_{h, \delta}^n$, error $e^n$ and network $R_{\delta}$ be defined as in the Definition \ref{def_network_predicted_traj}.
Then there exists a constant $C_T>0$ depending on $T$, independent of $\delta, \Delta t$, such that
\begin{align*}
     \norm{v_h^{N_t}-v_{h,\delta}^{N_t}}_{L^2(\Om)} = \norm{e^{N_t}}_M\le C_T\delta.
\end{align*}
\end{proposition}

\begin{proof}
By Lemma~\ref{lemma_numerical_operator_bounds} and \eqref{eq:kpp-exact-quadrature-consequences},
$ \norm{L_h}_{M\to M}\le1+\dt,
 \norm{C_h}_{W\to M}\le\dt,
 \norm{B}_{M\to W}=1,
 \sqrt{\Tr(W)}=|\Omega|^{1/2}.$
Therefore Lemma~\ref{lem:recursion} gives constants $c_0,c_1>0$, independent of $h$, $\dt$, and $\delta$, such that
\begin{equation}\label{eq:discrete-recursion}
 \norm{e^{n+1}}_M\le(1+c_0\dt)\norm{e^n}_M+c_1\dt\delta.
\end{equation}
Let $E_n:=\norm{e^n}_M$.  
Since $E_0=0$, repeated use of \eqref{eq:discrete-recursion} gives
\begin{align*}
    E_{N_t}
 \le
 c_1\dt\delta\sum_{\ell=0}^{N_t-1}(1+c_0\dt)^\ell =  c_1\dt\delta \frac{(1+c_0\dt)^{N_t}-1}{c_0\dt}, \text{ when } c_0>0.
\end{align*}
$N_t\dt=T$ and the elementary inequality $1+x\le e^x, x\geq 0$ gives 
$(1+c_0\dt)^{N_t}\le (e^{c_0\dt})^{N_t}=e^{c_0N_t\dt}=e^{c_0T}.$
Thus $ E_{N_t}
 \le
 \frac{c_1}{c_0}\left(e^{c_0T}-1\right)\delta
 =:C_T\delta.$
If $c_0=0$, the same conclusion follows from $E_{N_t}\le C_0T\delta$.  Finally, the mass norm and $L_2$ equivalence gives the result.
\end{proof}

\subsection{Numerical error}\label{sec:classical-error}
Let $v(T;v_0)$ denote the exact Fisher--KPP solution at terminal time $T$, and let $v_h^{N_t}(v_0)$ be the fully discrete finite element function generated by \eqref{eq:kpp-one-step-final} from the chosen finite element encoding of $v_0$.  
For the numerical error estimate below, we assume that the
chosen initial finite element encoding satisfies
$\|v_0-v_h^0(v_0)\|_{L^2(\Omega)}
\le C_{\rm init}h^m$,
uniformly over the admissible initial data, where
$C_{\rm init}$ is independent of $h$ and $v_0$.
The contribution of this initial approximation error is
included in the spatial error term below.
Under standard regularity assumptions on the solution, stability of the explicit scheme, and Lipschitz control of the reaction on the relevant bounded range, semilinear parabolic finite element theory gives an estimate of the form \cite{thomee2007galerkin},
\begin{align}
     \norm{v(T;v_0)-v_h^{N_t}(v_0)}_{L^2(\Om)}
 \leq
 C_{\rm num}\bigl(h^m+\dt+\eta_Q\bigr).
 \label{eqn_numerical_error}
\end{align}
Here $m$ is the spatial order in the $L^2(\Omega)$ norm; for $P_r$ elements and sufficiently smooth solutions one typically has $m=r+1$.  The term $\eta_Q$ denotes the consistency error due to quadrature.  In the Fisher--KPP square case, if the elementwise quadrature rule is exact through degree $3r$ on affine elements, then $\eta_Q=0$.

Combining \eqref{eqn_numerical_error} with accumulative error estimate Proposition~\ref{prop_accumulative_error} gives
\begin{align*}
     \norm{v(T;v_0)-v_{h,\delta}^{N_t}(v_0)}_{L^2(\Om)}
 \le
 C\bigl(h^m+\dt+\eta_Q+\delta\bigr).
\end{align*}
If the square quadrature is exact, then $\eta_Q=0$.
Besides, by the CFL $\dt\lesssim h^2$, this becomes
\begin{align*}
     \norm{v(T;v_0)-v_{h,\delta}^{N_t}(v_0)}_{L^2(\Om)}
 \le
 C\bigl(h^m+h^2+\delta\bigr).
\end{align*}

\subsection{Semi-network approximation error}
\label{sec_main_theorem_semi}

We now study the semi-network approximation as in Definition \ref{def_network_predicted_traj} where the coefficient is apprpximated by the iterative network trajectory $u_{h, \delta}^n$  while the finite element basis functions can be realized exactly by ReLU networks.  
This is the case for continuous $P_1$ Lagrange finite elements on affine simplicial meshes, since the corresponding basis functions are continuous piecewise affine functions. 
In this subsection, no additional basis-network approximation error is introduced.
We need the following additional assumptions.

\begin{assumption}
\label{assumption_fem_p1}
Let $\calA$ be a class of admissible initial data.  We assume the following.
\begin{enumerate}
\item $P_1$ finite element setting.
The spaces $V_h\subset H_0^1(\Omega)$ are continuous $P_1$ Lagrange finite element spaces on affine, shape-regular, quasi-uniform  and locally convex simplicial meshes.  
Thus each basis function $\phi_j$ is continuous and piecewise affine, and the number of degrees of freedom satisfies $N\sim h^{-d}$.

\item Initial-data encoding.
There is a fixed finite element encoder $\Enc:\calA\to\R^N$, for example nodal interpolation or $L^2$-projection, and both the numerical scheme and the network scheme start from $ u_h^0=u_{h,\delta}^0=\Enc(v_0).$
\end{enumerate}
\end{assumption}

\begin{remark}
For continuous $P_1$ elements on affine simplicial meshes, the finite element basis functions are continuous piecewise affine functions and hence can be represented exactly by ReLU networks.  
The result is established in \cite{he2020relu} and is summarized as the lemma below.
\begin{lemma}
\label{lemma_jinchao_xu}
On a locally convex simplicial mesh, each continuous $P_1$
nodal basis function admits an exact ReLU representation
with depth at most $C(1+\log_2 k_h)$ and width at most
$Ck_h$, where $C$ is independent of $h$.
\end{lemma}

This exact-realization argument does not directly apply to $Q_1$ elements in dimension $d\ge2$, because $Q_1$ basis functions are generally bilinear or multilinear on tensor-product elements, not piecewise affine.  
Thus $Q_1$ bases would require an additional approximation argument unless $d=1$, where $Q_1=P_1$.
This will be discussed in the next sections.
\end{remark}

Next, we establish the semi-network approximation theorem for exact FEM basis.
\begin{theorem}
\label{thm_semi_approximation}
Suppose Assumption~\ref{assumption_cfl} (CFL), \ref{assumption_quadrature} (exact quadrature), \ref{assumption_reference_bounded} (numerical bounded) and \ref{assumption_fem_p1} (FEM settings) hold.  
Assume also that the classical fully discrete $P_1$ finite element solution satisfies $ \norm{v(T;v_0)-v_h^{N_t}(v_0)}_{L^2(\Om)}
 \le
 C_{\rm num}\bigl(h^m+\dt\bigr)$
uniformly for $v_0\in\calA$.  
Let $0<\varepsilon<1/2$ be arbitrary, and $0<\delta<M_0^2$ be estimated later.
Then there exist: $h>0$, $\dt>0$, $N>0$, and one ReLU finite element coefficient network
$\NN_{h,\varepsilon}:\R^N\to\R^N$, 
and basis networks $\{\phi_{h, j}\}_{j = 1}^{N}:\mathbb{R}^d \rightarrow \mathbb{R}$
and the neural operator
\begin{equation}\label{eqn_semi_operator_structure}
 \calG_{h,\varepsilon}(v_0)(x)
 :=
 \sum_{j=1}^N [\NN_{h,\varepsilon}(\Enc(v_0))]_j\phi_{h, j}(x)
\end{equation}
such that $\textcolor{black}{\norm{v(T;v_0)-\calG_{h,\varepsilon}(v_0)}_{L^2(\Om)}
 \le
 C\bigl(h^m+\dt+\delta\bigr)
 }$, for any admissible initial condition $v_0$.
Here $C$ is independent of $h$, $\dt$, $\delta$, and $v_0\in\calA$. 
\medskip 

With additional $\dt\sim h^2$ from the CFL assumptions, and set $ \eta:=\min\{m,2\}.$
One can choose
$ h\sim\varepsilon^{1/\eta},
 \dt\sim h^2,
 \delta\sim\varepsilon$.
 Then $N \sim \varepsilon^{-d/\eta}$, and $N_t \sim h^{-2}$. 
 It follows $\norm{v(T;v_0)-\calG_{h,\varepsilon}(v_0)}_{L^2(\Om)}
 \lesssim \varepsilon$. 
Coefficient network $\mathcal{N}_{h, \varepsilon}$ depth and width can be chosen to satisfy
\begin{equation}\label{eq:depth-width-final}
 \operatorname{depth}(\mathcal{N}_{h, \varepsilon})
 \lesssim
 \varepsilon^{-2/\eta}(1+ \log(1/\varepsilon)),
 \operatorname{width}(\mathcal{N}_{h, \varepsilon})
 \lesssim
 \varepsilon^{-d/\eta}.
\end{equation}
\textcolor{black}{Each basis network $\phi_{h, j}$ can be chosen to have constant depth with width independent of $h, \varepsilon$.}
\end{theorem}

\begin{proof}
We have the numerical error  estimate $\norm{v(T;v_0)-v_h^{N_t}(v_0)}_{L^2}$, it follows to construct neural networks to approximate  $v_{h}^{N_t}(v_0)$ uniformly holds for all $v_0$.
This can be done by using networks to approximate the coefficients $u_{h}^{N_t}$ and construct basis $\phi_j$ separately.
\medskip

We first construct the coefficient network $\NN_{h,\varepsilon}$.  
$\NN_{h,\varepsilon}$ has an iterative structure while the each layer can be constructed by using the Definition \ref{def_network_predicted_traj}.
Specifically, given each time-step block's input $u\in\mathbb{R}^{N}$,
we have the block output as $u\longmapsto L_hu+C_hR_\delta(Bu).$
$L_h, C_h$ and $B$ are derived depending on the numerical scheme and can be exactly representable by affine layers in a ReLU network. 
The only approximation inside the coefficient network is the scalar clipped-square approximation $R_{\delta}$.  
Stacking this block $N_t$ times gives a ReLU coefficient network $\NN_{h,\varepsilon}:\R^N\to\R^N$ whose output is the coefficient vector $u_{h,\delta}^{N_t}$.

Now for the basis $\phi_j$.
By Lemma \ref{lemma_jinchao_xu},  $P_1$ basis functions $\phi_j$ can be exactly represented by ReLU networks denoted as $\phi_{h, j}$.
Hence there is no approximation error of the network realization of the basis.
The neural operator in \eqref{eqn_semi_operator_structure} is exactly the finite element reconstruction of the network-generated coefficient vector $ \calG_{h,\varepsilon}(v_0)(x)
 =
 \sum_{j=1}^N (u_{h,\delta}^{N_t})_j\phi_j(x)
 =
 v_{h,\delta}^{N_t}(x;v_0).$
 \medskip
 
The error estimate follows from the triangle inequality:
\begin{align*}
     \norm{v(T;v_0)-v_{h,\delta}^{N_t}(v_0)}_{L^2}
 \le
 \norm{v(T;v_0)-v_h^{N_t}(v_0)}_{L^2}
 +
 \norm{v_h^{N_t}(v_0)-v_{h,\delta}^{N_t}(v_0)}_{L^2}.
\end{align*}
The first term is bounded by the classical finite element estimate $ \norm{v(T;v_0)-v_h^{N_t}(v_0)}_{L^2}
 \le
 C_{\rm num}(h^m+\dt).$
The second term is bounded by Proposition~\ref{prop_accumulative_error}, specifically, we have  $ \norm{v_h^{N_t}(v_0)-v_{h,\delta}^{N_t}(v_0)}_{L^2}
 \le
 C_T\delta$.
Combining these two bounds gives the estimate.
Now for any $\varepsilon>0$, by the  CFL assumption $\dt\sim h^2$, set $\eta = \min\{m,2\}$, $\delta \sim \varepsilon$ and choose 
$h\sim \varepsilon^{1/\eta}$, 
gives $\sup_{v_0\in\calA}
 \norm{v(T;v_0)-\calG_{h, \varepsilon}(v_0)}_{L^2(\Om)}
 \le
 C\varepsilon.$
 \medskip 
 
Now let us estimate the depth and width of the coefficient network.
Since the structure is iterative, we consider one layer.
Let us first consider $R_{\delta} = \tilde{R}_{\delta}\circ \chi_{M_0}$. 
By Lemma~\ref{lemma_yarotsky_product}, the $\tilde{R}_{\delta}$ has depth
$D_{\rm sq}(\delta)\le c_0+c_1\log(M_0^2/\delta)$ and constant width.
The clipping network is implemented by ReLU and has constant depth and width, and the affine maps $L_h\in\mathbb{R}^{N\times N}$, $C_h\in\mathbb{R}^{N\times Q}$ and $B\in \mathbb{R}^{Q\times N}$ in each time-step block add only a constant number of layers.  
Therefore one time-step block has depth dominated by $\tilde{R}_{\delta}$, i.e. $D_{\rm block}(\delta)\le \widetilde c_0+\widetilde c_1\log(M_0^2/\delta)$
with $\tilde{c}_0, \tilde{c}_1$ are constants.
\medskip

Since $N_t=T/\dt\sim h^{-2}$, the coefficient network has depth $ \operatorname{depth}(\NN_{h,\varepsilon})
 \lesssim
 h^{-2}(1+\log(M_0^2/\delta)).$
With $h\sim\varepsilon^{1/\eta}$ and $\delta\sim\varepsilon$, this gives
\begin{align*}
     \operatorname{depth}(\NN_{h, \varepsilon})
 \lesssim
 \varepsilon^{-2/\eta}(1 + \log(1/\varepsilon)).
\end{align*}

For width, each block stores vectors of size comparable to $N$ and $Q$, and applies $Q$ scalar square subnets in parallel.  
Since the scalar square subnet has bounded width up to $N, Q \sim h^{-d}$ ($Q \sim N\times n_q$, where $n_q$ is the number of quadrature points which depends on fixed constants $r$ and $d$), we have
$ \operatorname{width}(\NN_{h,\varepsilon})\lesssim N+Q\lesssim h^{-d}.$
Using $h\sim\varepsilon^{1/\eta}$ gives
\begin{align*}
     \operatorname{width}(\NN_{h, \varepsilon})\lesssim\varepsilon^{-d/\eta}.
\end{align*}
This proves \eqref{eq:depth-width-final}.
For the basis network $\phi_{h, j}$, by Lemma \ref{lemma_jinchao_xu}, depth $\sim \log_2 k_h$ and width $\sim k_h$. 
Due to the uniform shape regular mesh assumption of $\mathcal{T}_h$, $k_h$ is of an order constant, which completes the proof.
\end{proof}

\begin{remark}[Parameter count versus width]
The width estimate in \eqref{eq:depth-width-final} counts the maximum hidden-layer width of the coefficient network.  It is not a bound on the number of affine parameters.  Dense matrices such as $M^{-1}B^TW$ may contain many entries, unless one further exploits sparsity, locality, mass lumping, iterative solvers, or structured implementations.   
\end{remark}

\subsection{Full approximation with nonlocal approximation}
\label{sec:main-theorem}
The coefficient network constructed above approximates the terminal finite
element coefficient vector. 
However the basis is constructed exactly using the ReLU network of constant size.
This is true only when linear $P_1$ FEM is used.
In general, the basis could be constructed by the network, leading to another approximation. 
In this section, we will study this general case.
To obtain an operator architecture in which the
spatial basis functions are also represented by neural networks, we retain
$\phi_j$ for the exact finite element basis function and introduce a ReLU
network
\begin{align*}
     \Phi_{j,h,\gamma}:\mathbb R^d\to\mathbb R,
\end{align*}
that approximates $\phi_j$.  Here $\gamma>0$ denotes the
accuracy used for each individual basis network.  It is distinct from the
square-network tolerance $\delta$ and the final target accuracy
$\varepsilon$.

\medskip
We additionally assume the following.
\begin{assumption}[Finite element setting and initial-data encoding]
\label{assumption_fem_mesh_nonlocal}
Let $\mathcal A$ be a class of admissible initial data. 
The spaces
$V_h=\operatorname{span}\{\phi_1,\ldots,\phi_N\}
\subset H_0^1(\Omega)$
are continuous Lagrange $P_r$ finite element spaces of fixed degree
$r\geq1$ on a conforming, shape-regular, quasi-uniform family
$\mathcal T_h$ of affine simplicial meshes of
$\Omega=[0,1]^d$. 
The dimension $d$ and degree $r$ are fixed
independently of $h$.
Additionally, we assume there is a fixed finite element encoder
$\mathcal{E}_h:\mathcal A\to\mathbb R^N$
such that, for every $v_0\in\mathcal A$, the numerical and network
schemes use the same initial coefficient vector $u_h^0=u_{h,\delta}^0 = \mathcal{E}_h(v_0).$
\end{assumption}

\textcolor{black}{The assumption implies the following standard FEM results, see for example \cite{ern2004theory}.}
\begin{lemma}
\label{lemma_fem_properties}
Under Assumption~\ref{assumption_fem_mesh_nonlocal}, there exist constants
$c_M,C_M,c_D,C_D,C_{\mathrm{loc}}>0$, independent of $h$, $j$ such that
\begin{align*}
    N\sim h^{-d},
\quad
c_Mh^d\|a\|_2^2
\le a^TMa
\le C_Mh^d\|a\|_2^2,
\quad\text{for all }a\in\mathbb R^N.
\end{align*}

For each nodal Lagrange basis function $\phi_j$, there exists a closed axis-aligned box
\begin{align*}
    \omega_{j,h}:=\operatorname{int}(\operatorname{supp}\phi_j)
\subset D_{j,h}
=
\prod_{i=1}^d[a_{j,i},b_{j,i}]
\end{align*}
such that $c_Dh\le b_{j,i}-a_{j,i}\le C_Dh, i=1, ... ,d,$
and $\sup_{x\in\mathbb R^d}
|\{j:x\in D_{j,h}\}|
\le C_{\mathrm{loc}}.$
\end{lemma}

The next lemma is to approximate the FEM basis.
Related neural network constructions for high-order finite elements include
the univariate ReLU approximations in
\cite{opschoor2020deep,opschoor2024deep},
the multidimensional tensor-product constructions in
\cite{marcati2023exponential}, and exact representations on simplicial
meshes using ReLU and squared-ReLU activations
\cite{he2026deep}.
Here we give a direct argument based on the localized support and
first-derivative scaling of the basis functions, without requiring
an explicit polynomial representation.

\begin{lemma}
\label{lemma_basis_network_nonlocal}
For finite element basis function $\phi_j$ satisfying Assumption \ref{assumption_fem_mesh_nonlocal} and every
$0<\gamma<1$, there exists a ReLU network
\begin{align*}
     \Phi_{j,h,\gamma}:\mathbb R^d\to\mathbb R,
\end{align*}
such that $ \norm{\phi_j-\Phi_{j,h,\gamma}}_{L^\infty(\Om)}
 \leq
 \gamma.$
Moreover,
\begin{align*}
 \operatorname{depth}(\Phi_{j,h,\gamma})
 &\lesssim
 C_{\rm FE}\bigl(1+\log(1/\gamma)\bigr),\\
 \operatorname{width}(\Phi_{j,h,\gamma})
 &\lesssim \gamma^{-d},\\
 \operatorname{size}(\Phi_{j,h,\gamma})
 &\lesssim
 C_{\rm FE}\gamma^{-d} \bigl(1+\log(1/\gamma)\bigr).
\end{align*}
The constant $C_{\rm FE}$ may depend on $d$, $r$, and the
shape-regularity constant, but is independent of $h$, $j$, and
$\gamma$.  The values of some affine weights may depend on
$h$.  No support-preservation property is asserted or required.
\end{lemma}
\begin{proof}
Due to the $h-$dependent bound on $\nabla \phi$, direct apply of Yarotsky is not working. The full argument is included in the Appendix \ref{sec_appendix_proof_of_fem_basis_approx}.
\end{proof}

The next lemma estimates the error contribution from approximating the FEM basis using the network.

\begin{lemma}
\label{lemma_basis_error_cumulation_nonlocal}
Assume Assumption \ref{assumption_fem_mesh_nonlocal} (FEM setting) holds, and let $\Phi_{j, h, \gamma}$ be the network in Lemma~\ref{lemma_basis_network_nonlocal}.
Then for every coefficient vector
$a=(a_1,\ldots,a_N)^T\in\mathbb R^N$, we have
\begin{align*}
\norm{
   \sum_{j=1}^N
   a_j\bigl(\phi_j-\Phi_{j,h,\gamma}\bigr)
 }_{L^2(\Om)}
 \leq
 C_{\rm bas}h^{-d}\gamma\norm{a}_M,
\end{align*}
where $C_{\rm bas}$ is independent of $h$, $a$, and
$\gamma$.
\end{lemma}

\begin{proof}
Set $ e_j:=\phi_j-\Phi_{j,h,\gamma}.$
By Lemma~\ref{lemma_basis_network_nonlocal}, we have $ \norm{e_j}_{L^\infty(\Om)}\leq\gamma,$
and hence $\norm{e_j}_{L^2(\Om)}
 \leq
 \abs{\Om}^{1/2}\gamma. 
 $
The triangle inequality gives
\begin{align*}
 \norm{\sum_{j=1}^Na_je_j}_{L^2(\Om)}
\leq
 \sum_{j=1}^N\abs{a_j}\norm{e_j}_{L^2(\Om)}
 \leq
 \abs{\Om}^{1/2}\gamma\norm{a}_1.
\end{align*}
By the finite-dimensional Cauchy--Schwarz inequality, $\norm{a}_1\leq\sqrt{N}\norm{a}_2.$
Assumption \ref{assumption_fem_mesh_nonlocal} and the consequent Lemma \ref{lemma_fem_properties} implies $ N\leq C_Nh^{-d}.$
Moreover, the mass-matrix spectral equivalence gives $c_Mh^d\norm{a}_2^2
 \leq
 a^TMa
 =
 \norm{a}_M^2, $ or, $ \norm{a}_2
 \leq
 c_M^{-1/2}h^{-d/2}\norm{a}_M.$
It follows that,
\begin{align*}
 \norm{\sum_{j=1}^Na_je_j}_{L^2(\Om)}
 &\leq
 C_{\rm bas} \gamma
 h^{-d/2}h^{-d/2}\norm{a}_M =
 C_{\rm bas}\,h^{-d}\gamma\norm{a}_M.
\end{align*}
\end{proof}

The basis networks are not required to preserve the supports of the finite element basis functions and may be nonzero elsewhere in $\Omega$.
The next theorem estimates the full operator approximation
without requiring the basis networks to preserve the
finite element supports.

\begin{theorem}
\label{thm_main_nonlocal}
Suppose Assumption~\ref{assumption_cfl} (CFL condition), \ref{assumption_quadrature} (exact quadrature), \ref{assumption_reference_bounded} (numerical bounded), and \ref{assumption_fem_mesh_nonlocal} (FEM setting) hold.  
Besides, assume the numerical error satisfies, 
$\sup_{v_0\in\mathcal A}
\|v(T;v_0)-v_h^{N_t}(v_0)\|_{L^2(\Omega)}
\le C_{\rm num}(h^m+\Delta t).$
Let $0<\varepsilon<1/2$ and let $0<\delta<M_0^2$, $0<\gamma<1$ to be determined later.
There exist scalers $N, N_t$, $h$, $\Delta t$ to be estimated later and 
\begin{enumerate}[label=(\roman*),leftmargin=2.5em]
\item a ReLU coefficient network $ \NN_{h,\delta}:\mathbb R^N\to\mathbb R^N,$
obtained by iterating a block $u \rightarrow L_hu + C_hR_{\delta}(Bu)$ $N_t \sim h^{-2}$ times, where $R_{\delta}$ is a square clipping network,
\item ReLU basis networks $\Phi_{j,h,\gamma}:\mathbb R^d\to\mathbb R,
 j=1,\ldots,N,$
\end{enumerate}
such that the neural operator
$\mathcal{G}_{h,\varepsilon}(v_0)(x)
 :=
 \sum_{j=1}^N
 \bigl[\NN_{h,\delta}(\Enc(v_0))\bigr]_j
 \Phi_{j,h,\gamma}(x)$
satisfies the error estimates
\begin{align}
\label{eq:main-error-with-basis-nonlocal}
    \sup_{v_0\in\calA}
 \norm{
 v(T;v_0)-\calG_{h,\varepsilon}(v_0)
 }_{L^2(\Om)}
 \leq
 C\bigl(h^m+\dt+\delta+h^{-d}\gamma\bigr) \leq C\varepsilon.
\end{align}
The constant $C$ is independent of $h$, $\dt$, $\delta$,
$\gamma$, and $v_0$.
Additionally by the CFL and $\eta:=\min\{m,2\}.$
One can choose
\begin{align*}
     h\sim\varepsilon^{1/\eta},
 \quad
 \dt\sim h^2,
 \quad
 \delta\sim\varepsilon,
 \quad
 \gamma\sim\varepsilon h^d
 \sim
 \varepsilon^{1+d/\eta},
 \quad
 N \sim \varepsilon^{-d/\eta}
\end{align*}

Coefficient network $\NN_{h,\delta}(\Enc(v_0))$ has depth $\sim\delta^{-2/\eta}\log(1/\delta)$ and width $\delta^{-d/\eta}$.
Each $\Phi_{j,h,\gamma}$ has depth $\log(1/\gamma)$ and width $\gamma^{-d}$.
The dominating depth of $\mathcal{G}_{h, \varepsilon}$ comes from $\NN_{h,\delta}$ and the dominating width comes from parallel concatenating all $N$ basis network $\Phi_{j,h,\gamma}$, and we have
\begin{equation*}
 \operatorname{depth}(\calG_{h, \varepsilon})
 \lesssim
 \varepsilon^{-2/\eta}\log(1/\varepsilon),
 \quad
 \operatorname{width}(\calG_{h, \varepsilon})
 \lesssim
 \varepsilon^{-d/\eta-d-d^2/\eta}.
\end{equation*}

\end{theorem}

\begin{proof}
We have the numerical  error estimate $\|v(T; v_0) - v_h^{N_t}(v_0)\|_{L^2}$.
Now we need to construct networks to construct a network realization $v_{h, \delta}^{N_t}(v_0) = \sum_{j = 1}^N (u_{h, \delta}^{N_t})_j\phi_j$ to approximate $v_h^{N_t}(v_0)$.
For the coefficient $(u_{h, \delta}^{N_t})_j$, we use network $[\NN_{h,\delta}(\Enc(v_0))]_j$ and the construction is the same as Theorem \ref{thm_semi_approximation} (though the $P_r$ assumptions are different however this does not affect the coefficient construction) and we omit the details. 
For the basis $\phi_j$, the approximation by $\Phi_{j,h, \gamma}$ in $L^{\infty}$ norm with $\gamma$ accuracy is studied in Lemma \ref{lemma_basis_network_nonlocal}.

Now add and subtract the reference finite element solution
and the exact-basis reconstruction of the coefficient-network output:
\begin{align}
 v(T;v_0)
 -\mathcal{G}_{h,\varepsilon}(v_0)
 ={}&
 \bigl(v(T;v_0)-v_h^{N_t}(v_0)\bigr)
 +
 \sum_{j=1}^N
 \left[(u_h^{N_t})_j-(u_{h,\delta}^{N_t})_j\right]\phi_j
 +
 \sum_{j=1}^N
 (u_{h,\delta}^{N_t})_j
 \bigl(\phi_j-\Phi_{j,h,\gamma}\bigr).
 \label{eq:three-way-error-splitting-nonlocal}
\end{align}
The first term is bounded by
the classical numerical error estimate.  
For the second term, the exact
mass-matrix identity and Proposition~\ref{prop_accumulative_error} give
\begin{align}
 \norm{
 \sum_{j=1}^N
 \left[(u_h^{N_t})_j-(u_{h,\delta}^{N_t})_j\right]\phi_j
 }_{L^2(\Om)}
 =
 \norm{u_h^{N_t}-u_{h,\delta}^{N_t}}_M
 \leq
 C_T\delta.
 \label{eq:coefficient-network-terminal-error-nonlocal}
\end{align}

It remains to bound the basis-network contribution.  
Since $B^TWB=M$ by the Assumption \ref{assumption_quadrature},
and the bounded numerical solutions Assumption \ref{assumption_reference_bounded} imply
\begin{align*}
 \norm{u_h^{N_t}}_M^2
 =
 \norm{Bu_h^{N_t}}_W^2
 =
 \sum_{q=1}^Q\omega_q\abs{(Bu_h^{N_t})_q}^2
\leq
 M_*^2\sum_{q=1}^Q\omega_q
 =
 M_*^2\abs{\Om}.
\end{align*}
Next using \eqref{eq:coefficient-network-terminal-error-nonlocal}, we have 
$ \norm{u_{h,\delta}^{N_t}}_M = \norm{u_{h,\delta}^{N_t} + u_h^{N_t} - u_h^{N_t} }_M
 \leq
 M_*\abs{\Om}^{1/2}+C_T\delta.$
Applying Lemma~\ref{lemma_basis_error_cumulation_nonlocal} with
$a=u_{h,\delta}^{N_t}$ gives
\begin{align*}
 \norm{
 \sum_{j=1}^N
 (u_{h,\delta}^{N_t})_j
 \bigl(\phi_j-\Phi_{j,h,\gamma}\bigr)
 }_{L^2(\Om)}
\leq
 C_{\rm bas}h^{-d}\gamma
 \norm{u_{h,\delta}^{N_t}}_M
 \leq
 C_{\rm bas}h^{-d}\gamma
 \left(M_*\abs{\Om}^{1/2}+C_T\delta\right).
\end{align*}
Since $0<\delta<M_0^2$ and $M_0$ is fixed independently
of the approximation parameters,
$M_*|\Omega|^{1/2}+C_T\delta
\le M_*|\Omega|^{1/2}+C_TM_0^2$.
Taking the $L^2(\Om)$ norm in
\eqref{eq:three-way-error-splitting-nonlocal}, applying the triangle
inequality, and combining the preceding estimates proves
\eqref{eq:main-error-with-basis-nonlocal}.

Next, by the CFL assumption we have $ \norm{
 v(T;v_0)-\calG_{h,\varepsilon}(v_0)
 }_{L^2(\Om)}
 \leq
 C\bigl(h^m+h^2+\delta+h^{-d}\gamma\bigr).$
With $\eta=\min\{m,2\}$, and we choose the parameters so that each error contribution
is $O(\varepsilon)$.
Then $ h^m+h^2\lesssim\varepsilon,
 \delta\lesssim\varepsilon,
 h^{-d}\gamma\lesssim\varepsilon.$
 Or $h \sim \varepsilon^{1/\eta}$, $\gamma = \varepsilon^{1+d/\eta}$.

\medskip
We finally estimate the size of the network.  
As established in Theorem \ref{thm_semi_approximation}, coefficient network $\NN_{h,\delta}(\Enc(v_0))$ is constructed iterating the block $u \rightarrow L_hu + C_hR_{\delta}(Bu)$ $N_t \sim h^{-2}$ times, consequently,
\begin{align*}
     &\operatorname{depth}(\NN_{h,\delta})
 \lesssim
 h^{-2}\bigl(1+\log(1/\delta)\bigr)\sim \varepsilon^{-2/\eta}\log(\varepsilon^{-1}),\\
 &\operatorname{width}(\NN_{h,\delta})
 \lesssim
 N+Q
 \lesssim
 h^{-d}\sim \varepsilon^{-d/\eta}.
\end{align*}
By Lemma~\ref{lemma_basis_network_nonlocal}, all $N$ basis networks can be
evaluated in parallel with
\begin{align*}
\operatorname{depth}
 \bigl(\{\Phi_{j,h,\gamma}\}_{j=1}^N\bigr)
 \lesssim
 1+\log(1/\gamma) \sim \left(1+\frac d\eta\right)\log(1/\varepsilon).
\end{align*}
\textcolor{black}{The width for one $\Phi_{j, h, \gamma}$ is $\lesssim \gamma^{-d}$}.
Due to the parallel concatenation of all $\Phi_{j,h,\gamma}$,
the width is $N\cdot \text{ width of } \Phi_{j,h,\gamma}$, which gives,
\begin{align*}
    \text{width}\sim h^{-d}\gamma^{-d} \sim  \varepsilon^{-d/\eta} \varepsilon^{-d-d^2/\eta}.
\end{align*}

\end{proof}

\begin{remark}[\textcolor{black}{Tensor-product $Q_r$ elements}]
The theorem is written for continuous $P_r$ Lagrange elements on affine
simplicial meshes.  The same error proof applies to continuous $Q_r$
elements on affine tensor-product meshes when the basis approximation lemma
is available for that element family and the quadrature is exact on
$\mathbb Q_{3r}(K)$.  In that case the number of local degrees of freedom is
$(r+1)^d$, while the global scalings $N\sim h^{-d}$ and
$Q\sim h^{-d}$ and the final rates are unchanged for fixed $r$.
\end{remark}

\subsection{Full approximation case with localized basis networks}
\label{sec_full_approximation_local}
In the last section, we consider using ReLU network to approximate the basis and derive the analysis.
However, the network may not present local properties even if the basis itself has a local support.

\subsubsection{Strict support-preserving approximation}
In this section, we will derive a new network approximation to FEM basis by enforcing the local support.
Similar results see \cite{lanthaler2022nonlinear,ali2021approximation, yarotsky2017error}.
Due to the localized argument, the rate will also be improved.

\begin{lemma}
\label{lemma_support_preserving_approximation}
Let $ D=\prod_{i=1}^d[a_i,b_i]\subset\mathbb R^d$
be a closed box with nonempty interior, and let $n\geq1$ be an integer.
Suppose that $f\in W^{n,\infty}(\mathbb R^d)$ and
$\operatorname{supp}(f)\subseteq D$.  
Then for any $0<\varepsilon<1$, there exists a scaler $C_f>0$, depending on $D$, $d$, $n$ and an upper bound for $\|f\|_{W^{n,\infty}(\mathbb R^d)}$, 
ReLU networks $\widetilde f_\varepsilon:\mathbb R^d\to\mathbb R$ such that
\begin{align*}
 \norm{f-\widetilde f_\varepsilon}_{L^\infty(\mathbb R^d)}
 \leq \varepsilon, \text{ and }
 \operatorname{supp}(\widetilde f_\varepsilon)\subseteq D.
\end{align*}
Moreover, the network can be chosen with same size as usual,
\begin{align*}
 \operatorname{depth}(\widetilde f_\varepsilon)
 &\lesssim C_f\bigl(1+\log(1/\varepsilon)\bigr),\\
 \operatorname{width}(\widetilde f_\varepsilon)
 &\lesssim \varepsilon^{-d/n},\\
 \operatorname{size}(\widetilde f_\varepsilon)
 &\lesssim
 C_f\varepsilon^{-d/n}\bigl(1+\log(1/\varepsilon)\bigr).
\end{align*}
Here size is the number of nonzero affine weights.
\end{lemma}
 The proof uses the techniques of creating a mask with local support and Yarotsky's product network property where the network returns zero when one input is zero.
 We give the proof the Appendix \ref{sec_proof_of_local_approx_lemma}.

\begin{remark}[Scope of the localization lemma]
\label{rem:scope-support-local}
The multiplication subnet must be accurate on a bounded interval containing
the range of the raw approximant; a multiplier stated only on $[0,1]^2$ is
not sufficient for a general signed function.  The exact zero property is
global and is what enforces the support restriction.  The buffer width must
also be small enough that $D_\kappa$ is nonempty. 
These points affect
constants and weight magnitudes but not the stated asymptotic depth and size.
\end{remark}

\subsubsection{Localized neural approximations of finite element bases}
Next, we study the approximation to FEM basis using the localized network.
\begin{lemma}
\label{lemma_basis_network_local}
Suppose Assumption~\ref{assumption_fem_mesh_nonlocal} (FEM setting) holds.  
For any $0<\gamma<1$ and basis function $\phi_j$, there is a
ReLU network $ \Phi_{j,h,\gamma}:\mathbb{R}^d\to\mathbb{R}$
satisfying
\begin{align*}
 \norm{\phi_j-\Phi_{j,h,\gamma}}_{L^\infty(\Om)}
 &\leq\gamma,\\
 \operatorname{supp}(\Phi_{j,h,\gamma})
 &\subseteq D_{j,h}.
\end{align*}
Moreover, one can estimate the network size,
\begin{align*}
 \operatorname{depth}(\Phi_{j,h,\gamma})
 &\lesssim
 C_{\rm bas}\bigl(1+\log(1/\gamma)\bigr),\\
 \operatorname{width}(\Phi_{j,h,\gamma})
 &\lesssim \gamma^{-d},\\
 \operatorname{size}(\Phi_{j,h,\gamma})
 &\lesssim
 C_{\rm bas}\gamma^{-d}
 \bigl(1+\log(1/\gamma)\bigr),
\end{align*}
where $C_{\rm bas}$ is independent of $h$, $j$, and
$\gamma$.  The numerical values of some affine weights may
depend on $h$.
\end{lemma}

\begin{proof}
Extend $\phi_j$ by zero outside $\Omega$, and denote the
extension by $\bar\phi_j$. Continuity and the homogeneous
boundary condition imply that
$\bar\phi_j\in W^{1,\infty}(\mathbb R^d)$.
The finite element scaling estimate gives
$\|\bar\phi_j\|_{L^\infty}
+h\|\nabla\bar\phi_j\|_{L^\infty}\le C$,
uniformly in $h$ and $j$.

Write
$D_{j,h}=\prod_{i=1}^d[a_{j,i},a_{j,i}+\ell_{j,i}]$,
where $\ell_{j,i} \sim h$, and define the affine bijection
$F_{j,h}:\mathbb R^d\to\mathbb R^d$ by
$(F_{j,h}(y))_i=a_{j,i}+\ell_{j,i}y_i$.
Set $\psi_{j,h}=\bar\phi_j\circ F_{j,h}$.
Then
$\operatorname{supp}(\psi_{j,h})\subseteq[0,1]^d$.
Moreover,
$\partial_{y_i}\psi_{j,h}
=\ell_{j,i}(\partial_{x_i}\bar\phi_j)\circ F_{j,h}$
almost everywhere, so
$\|\psi_{j,h}\|_{W^{1,\infty}(\mathbb R^d)}\le C$
uniformly in $h$ and $j$.
\medskip

Apply Lemma~\ref{lemma_support_preserving_approximation}
with $n=1$ on the fixed box $[0,1]^d$.
It provides a ReLU network $\widehat\Phi_{j,h,\gamma}$
with uniform error at most $\gamma$, support contained
in $[0,1]^d$, and the stated complexity bounds.
Define
$\Phi_{j,h,\gamma}
=\widehat\Phi_{j,h,\gamma}\circ F_{j,h}^{-1}$.
Since $F_{j,h}$ is bijective, the approximation error
remains at most $\gamma$, and
$\operatorname{supp}(\Phi_{j,h,\gamma})\subseteq D_{j,h}$.
The affine transformation preserves the complexity bounds
up to uniform constants, although its weights may depend
on $h$.
\end{proof}

Next, we will estimate the error contributed due to the local approximation of basis.
This lemma improves the former Lemma \ref{lemma_basis_error_cumulation_nonlocal} where the basis network has no local support.
Note that a factor of $h^{-d}$ is gone.

\begin{lemma}
\label{lem_basis_error_accumulation_local}
Suppose Assumption~\ref{assumption_fem_mesh_nonlocal} (FEM setting) holds and let the basis
networks be as in Lemma~\ref{lemma_basis_network_local}.  Then, for every
$a=(a_1,\ldots,a_N)^T\in\mathbb R^N$,
\begin{align*}
     \norm{
 \sum_{j=1}^N
 a_j\bigl(\phi_j-\Phi_{j,h,\gamma}\bigr)
 }_{L^2(\Om)}
 \leq
 C_{\rm acc}\gamma\norm{a}_M,
\end{align*}
where $C_{\rm acc}$ is independent of $h$, $a$, and
$\gamma$.
\end{lemma}

\begin{proof}
Set $e_j:=\phi_j-\Phi_{j,h,\gamma}.$
By Lemma \ref{lemma_basis_network_local}
Then $\norm{e_j}_{L^\infty(\Om)}\leq\gamma,
 \operatorname{supp}(e_j)\subseteq D_{j,h}.$
Now for any fixed $x\in \Omega$, by Lemma \ref{lemma_fem_properties}, there are $C_{\operatorname{loc}}$ number of boxes that contain $x$, and therefore $C_{\operatorname{loc}}$ error terms that are active at $x$,
\begin{align*}
     \abs{\sum_{j=1}^N a_je_j(x)}^2
 \leq
 C\gamma^2
 \sum_{j:x\in D_{j,h}}\abs{a_j}^2\leq C\gamma^2\sum_{j = 1}^N |a_j|^2 1_{D_{j, h}}(x).
\end{align*}
Now consider the $L^2$ norm and  integrateover $\Omega$ gives,
\begin{align*}
\left\|\sum_{j=1}^N a_je_j\right\|_{L^2(\Omega)}^2
&\leq
C_{\rm loc}\gamma^2
\sum_{j=1}^N |a_j|^2 \int_{\Omega} 1_{D_{j, h}} dx \leq C\gamma^2h^d\|a\|_2^2 \le    C\gamma^2\|a\|_M^2.
\end{align*}
 where we use 
$\abs{D_{j,h}}\leq Ch^d$
and the standard mass-norm and Euclidean norm equivalence in Lemma \ref{lemma_fem_properties}.
Combining the two estimates and taking square roots yield the results.
\end{proof}

\subsubsection{Full operator approximation}
Now we give the error estimate when basis network has a support property.
\begin{theorem}
\label{thm_main_local}
Suppose Assumptions~\ref{assumption_cfl} (CFL), \ref{assumption_quadrature} (exact quadrature),
\ref{assumption_reference_bounded} (numerical bounded), and
\ref{assumption_fem_mesh_nonlocal} (FEM setting) hold.
Besides, assume the numerical error satisfies, 
$\sup_{v_0\in\mathcal A}
\|v(T;v_0)-v_h^{N_t}(v_0)\|_{L^2(\Omega)}
\le C_{\rm num}(h^m+\Delta t).$
Let $0<\varepsilon<1/2$ be arbitrary and $0<\delta<M_0^2$, $0<\gamma<1$ to be determined later.
There exist scalars $N, N_t, h, \Delta t$
\begin{enumerate}[label=(\roman*),leftmargin=2.5em]
\item a ReLU coefficient network $ \NN_{h,\delta}:\mathbb R^N\to\mathbb R^N,$
obtained by iterating a block $u \rightarrow L_hu + C_hR_{\delta}(Bu)$ $N_t \sim h^{-2}$ times, where $R_{\delta}$ is a square clipping network,
\item localized ReLU basis networks $\Phi_{j,h,\gamma}:\mathbb R^d\to\mathbb R,
 j=1,\ldots,N,$
\end{enumerate}
such that the operator
$ \calG_{h,\varepsilon}(v_0)(x)
 :=
 \sum_{j=1}^N
 \bigl[\NN_{h,\delta}(\Enc(v_0))\bigr]_j
 \Phi_{j,h,\gamma}(x)$
satisfies the error estimate, 
\begin{equation}\label{eq:main-error-with-basis-local}
 \sup_{v_0\in\calA}
 \norm{
 v(T;v_0)-\calG_{h,\varepsilon}(v_0)
 }_{L^2(\Om)}
 \leq
 C\bigl(h^m+\dt+\delta+\gamma\bigr).
\end{equation}
The constant $C$ is independent of $h$, $\dt$, $\delta$,
$\gamma$, and $v_0\in\calA$.

By the CFL $\dt\sim h^2$, and set $ \eta:=\min\{m,2\}.$
One may choose
\begin{align*}
h\sim\varepsilon^{1/\eta},
 \quad
 N \sim \varepsilon^{-d/\eta},
 \quad
 \dt\sim h^2,
 \quad
 \delta\sim\varepsilon,
 \quad
 \gamma\sim\varepsilon,
\end{align*}
then
$ \sup_{v_0\in\calA}
 \norm{v(T;v_0)-\calG_{h, \varepsilon}(v_0)}_{L^2(\Om)}
 \leq C\varepsilon.$
 Coefficient network 
$\NN_{h,\delta}(\Enc(v_0))$ has depth:
$\varepsilon^{-2/\eta}\log(1/\varepsilon),$ width $\varepsilon^{-d/\eta}$.
Each localized basis networks $\Phi_{j,h,\gamma}$ has depth estimate
$ \log(1/\varepsilon),
$ and width
$\varepsilon^{-d}.$
The dominating depth of $\mathcal{G}_{h, \varepsilon}$ comes from $\NN_{h,\delta}$ and the dominating width comes from all $\Phi_{j,h,\gamma}$. 
Specifically, we have,
 \begin{equation*}
 \operatorname{depth}(\calG_\varepsilon)
 \lesssim
 \varepsilon^{-2/\eta}\log(1/\varepsilon),
 \quad
 \operatorname{width}(\calG_\varepsilon)
 \lesssim
 \varepsilon^{-d/\eta-d}.
\end{equation*}
 
\end{theorem}

\begin{proof}
The proof follows Theorem \ref{thm_main_nonlocal}.
However, by Lemma \ref{lem_basis_error_accumulation_local},
\begin{align*}
 \norm{
 \sum_{j=1}^N
 (u_{h,\delta}^{N_t})_j
 \bigl(\phi_j-\Phi_{j,h,\gamma}\bigr)
 }_{L^2(\Om)} \leq
 C_{\operatorname{acc}}\gamma
 \left(M_*\abs{\Om}^{1/2}+C_T\delta\right),
\end{align*}
where $h^{-d}$ factor on the right hand side is gone.
Set $\dt\sim h^2$ and for $0<\delta<M_0^2$, the error estimate becomes
\begin{align*}
     \sup_{v_0\in\calA}
 \norm{v(T;v_0)-\calG_{h,\varepsilon}(v_0)}_{L^2(\Om)}
 \leq
 C\bigl(h^m + \Delta t + \delta+\gamma\bigr).
\end{align*}
To estimate the all quantities, 
we choose the parameters so that each error contribution
is $\sim \varepsilon$.
The rest of the proof follows exactly as before.
\end{proof}

\begin{remark}[What localization improves]
\label{rem:local-versus-nonlocal}
The localized and nonlocalized arguments both yield a final
$L^2(\Om)$-error of order $\varepsilon$; locality does not change the
convergence order of the underlying finite element and time discretizations.
Its direct analytic improvement is $ h^{-d}\gamma\rightarrow
 \gamma,$ 
 where $\gamma$ is the error to in the FEM basis approximation.
Thus the nonlocalized argument requires
$\gamma\sim\varepsilon h^{d}$, whereas the localized argument only
requires $\gamma\sim\varepsilon$.

Consequently, the localized total basis-network size $\sim \left(
 \varepsilon^{-
d/\eta-d}\log(1/\varepsilon)
 \right),$
whereas the nonlocalized tolerance leads to $\sim \left(
 \varepsilon^{-d/\eta-d-d^2/\eta}\log(1/\varepsilon)
 \right).$
Accordingly, localization saves a factor $\varepsilon^{-d^2/\eta}$ under
that common construction. 
\end{remark}

\begin{remark}[Tensor-product $Q_r$ elements]
The same argument applies to continuous $Q_r$ Lagrange elements
on conforming, shape-regular, quasi-uniform affine
tensor-product meshes, with fixed degree $r$.
Use the affine images of a fixed positive quadrature rule
that is exact on $\mathbb Q_{3r}$ on the reference cube.
Then $N,Q\sim h^{-d}$, and the resulting rates are unchanged.
For both $P_r$ and $Q_r$ elements, the construction guarantees
support contained in $D_{j,h}$; this box need not coincide
with the exact support of the finite element basis function.
\end{remark}

\section{Numerical validation}
\label{sec_numerical}

In this section, we conduct the numerical tests.
We will study three different PDEs.
For each PDE, we will compare the following four models. 
(1) The proposed model which is based on Figure \ref{fig_full_nn} and \ref{fig_one_layer}. 
(2) The FERN model \cite{zhang2025finite}, this model is also designed based on \eqref{eqn_cc_neural_operator}. 
It uses the ReLU network to construct local learnable FEM basis (hat function) in an exact way, while use separate networks one for each coefficients.
(3) Proposed method plus FERN, or we will use the proposed method to construct the coefficients network $\mathcal{N}_{h, \delta}$ to generate the coefficients meanwhile using the ReLU network to construct the exact linear learnable FEM basis.
(4) Standard DeepONet  \cite{lu2021learning} which is also designed based on \eqref{eqn_cc_neural_operator}, here we consider the stack version which uses the separate branch network for each coefficient.
In contrast, the unstacked version DeepONet uses a centralized network $\mathbb{R}^{N}\rightarrow \mathbb{R}^{N}$ to generate all coefficients simultaneously. 
Since the current scaling-law-based convergence theory naturally corresponds to the stacked architecture, we use the stacked DeepONet as our baseline. Its convergence properties have also been studied in detail in \cite{liu2024neural}.
Notably, the FERN methods use two parameters to construct one FEM basis and will return a network with much fewer network parameters.

\textbf{Network and training settings.} We now give the detailed proposed network structure used in our experiments. 
The network consists of a coefficient network and a basis network.
The coefficient network takes 32 sampled values of the initial condition
and applies four successive blocks of the form
$v\mapsto A_\ell v+C_\ell\tanh(G_\ell v+b_\ell)$,
each with state dimension 32 and independently trainable parameters.
The basis network takes the spatial coordinate $x$ and uses six fully
connected hidden layers, each with 100 neurons and $\tanh$ activation,
followed by a linear output layer producing 32 basis values.
The predicted terminal solution is obtained by taking the inner product
of the coefficient-network and basis-network outputs.

All models are trained using the Adam optimizer for 4,000
epochs with an initial learning rate of $5\times10^{-3}$.
We minimize the root mean square error between the predicted and reference
terminal solutions, using all training pairs in each optimization step.
The learning rate is updated after each epoch using a cosine annealing
scheduler with $T_{\max}=3800$ and a minimum learning rate of $10^{-4}$.

\textbf{General conclusion from the numerical results. }
Tables~\ref{table_kpp_results}, \ref{table_aggre_diff_example},
and \ref{table_ac_example} demonstrate the effectiveness of the
proposed architecture, which combines a coefficient network
derived from numerical time stepping with either global or
localized basis networks. With 32 basis functions in all models,
the proposed global-basis model achieves the lowest mean relative
errors across the three PDEs, namely $0.08\%$, $0.28\%$, and
$0.54\%$, while using approximately $16\%$ fewer parameters than
DeepONet. The localized version further reduces the parameter
count by approximately $80$--$85\%$ relative to the global version,
at the cost of higher approximation errors. Although it does
not consistently outperform the baselines in accuracy, it uses
the fewest parameters in every experiment. These results highlight
the contribution of our approach: a numerical-scheme-based
network design that accommodates different basis representations
and provides a practical tradeoff between accuracy and model size.

\subsection{Allen Cahn equation}
In this section, we study the Allen--Cahn equation,
\begin{align*}
    u_t-\nu u_{xx}=\lambda(u-u^3),
    \quad x\in[0,1], \quad t\in[0,2],
\end{align*}
where $\nu=10^{-3}$ and $\lambda=1$, subject to the homogeneous
Dirichlet boundary conditions $u(0,t)=u(1,t)=0$.
The target operator is the mapping from the initial condition
to the solution at the terminal time $T=2$.
We generate the initial condition by sampling the free parameters
$a_1$ and $a_2$ independently and uniformly from $[-0.5,0.5]$
according to
\begin{align*}
    u_0(x)=\tanh\left(
    a_1\sin(\pi x)+a_2\sin(2\pi x)
    \right).
\end{align*}
These initial conditions satisfy the boundary conditions and
take values in $(-1,1)$.
We present two example solutions in Figure~\ref{fig_ac_sol_demon}.
\begin{figure}[H]
    \centering
    \includegraphics[scale = 0.4]{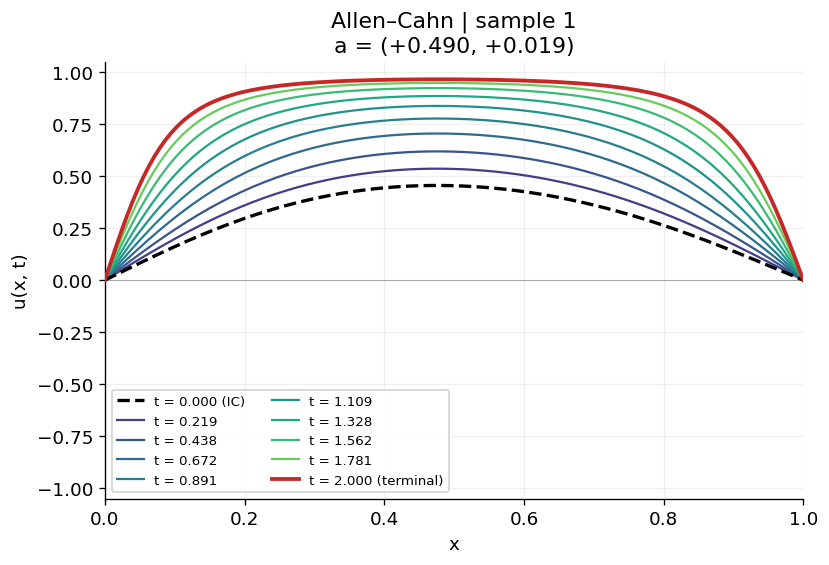}
    \includegraphics[scale = 0.4]{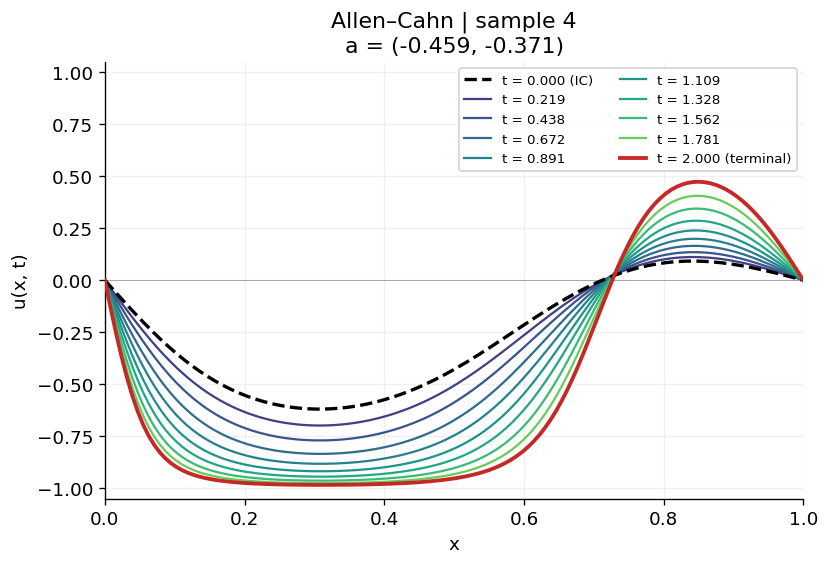}
    \caption{Demonstration of two Allen-Cahn solutions. The initial condition is generated using $u_0(x)=\tanh\left(
    a_1\sin(\pi x)+a_2\sin(2\pi x)
    \right)$, with $a_1, a_2$  specified as in the Figure.}
    \label{fig_ac_sol_demon}
\end{figure}
We present the numerical results in
Table~\ref{table_ac_example}.

\begin{table}[H]
\centering
\small
\setlength{\tabcolsep}{4pt}
\renewcommand{\arraystretch}{1.15}
\begin{tabular}{|l|c|c|c|c|}
\hline
& \shortstack{\textbf{Proposed with}\\\textbf{global basis}}
& \shortstack{\textbf{Proposed with}\\\textbf{local basis}}
& \textbf{DeepONet}
& \textbf{FERN} \\
\hline
\textbf{Relative error (\%)}
& $0.54 \pm 0.25$
& $0.89 \pm 0.17$
& $0.6 \pm 0.21$
& $0.72 \pm 0.18$ \\
\hline
\textbf{\# parameters}
& $63,316$
& $12,480$
& $75,660$
& $21,824$ \\
\hline
\textbf{\# basis}
& $32$
& $32$
& $32$
& $32$ \\
\hline
\end{tabular}
\caption{Results for Allen Cahn example.}
\label{table_ac_example}
\end{table}

\begin{figure}[H]
    \centering
    \includegraphics[scale = 0.5]{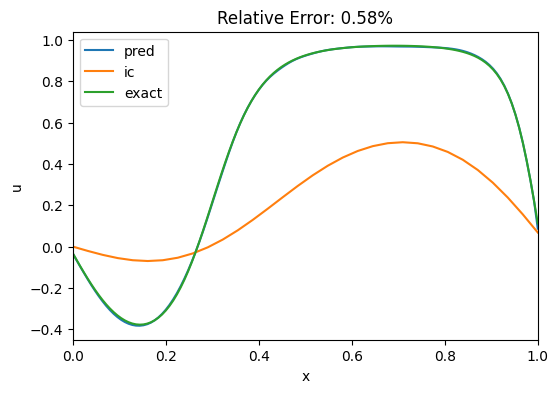}
    \includegraphics[scale = 0.5]{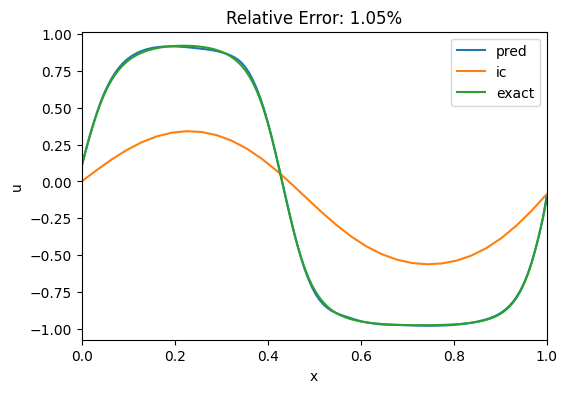}
    \caption{Predictions using the trained proposed model.}
    \label{fig_ac_pred}
\end{figure}

\subsection{KPP example}
In this section, we study the Fisher--KPP reaction-diffusion equation,
\begin{align*}
    u_t-u_{xx}=u-u^2, \quad x\in[0,1], t\in[0,T],
\end{align*}
subject to the homogeneous Dirichlet boundary conditions $u(0,t)=u(1,t)=0$. The target operator is the mapping from the initial condition to the solution at the terminal time $T$. We generate the initial condition by sampling the free parameters $a_1$ and $a_2$ independently and uniformly from $[-1,1]$ according to
$$
u_0(x)=\frac{1}{4}\sin(\pi x)\left(2+a_1\sin(2\pi x)+a_2\cos(2\pi x)\right).
$$
We present the results two demonstration in Figure \ref{fig_kpp_sol} and the results in Table \ref{table_kpp_results}.
\begin{figure}[H]
    \centering
    \includegraphics[scale = 0.4]{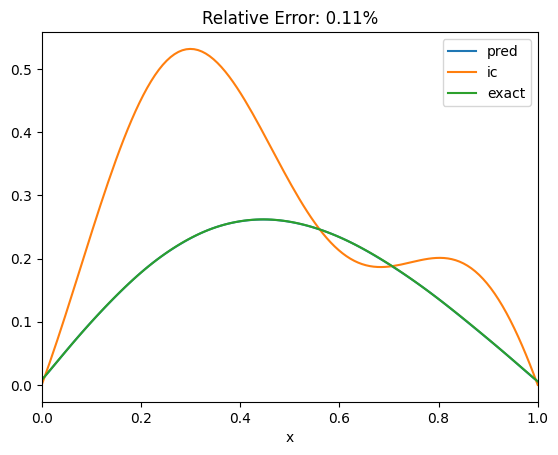}
    \includegraphics[scale = 0.4]{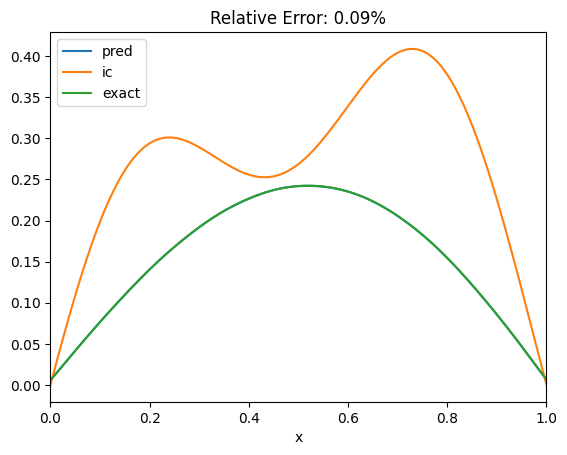}
    \caption{Demonstration of KPP predictions.}
    \label{fig_kpp_sol}
\end{figure}
Specifically, we compare four models. 
(1) The proposed model with the architecture in Section \ref{sec_nn_structure} and a global fully connected basis. 
(2) The proposed model with the same architecture, but with the basis functions constructed exactly using ReLU networks to represent the $P_1$ FEM basis functions, as in \cite{zhang2025finite}; see also Lemma~\ref{lemma_jinchao_xu}. 
(3) The standard DeepONet. 
(4) FERN, which is designed based on \eqref{eqn_cc_neural_operator} and also uses exactly constructed, learnable FEM hat functions as the basis, while the corresponding basis coefficients retain the standard DeepONet design and are learned by separate fully connected neural networks. 
Notably, models (2) and (4) use the same type of FEM-based output basis, while their main difference lies in the coefficient network: model (2) adopts our newly designed coefficient network derived from the numerical scheme, whereas FERN uses the standard DeepONet coefficient-network architecture.

\begin{table}[H]
\centering
\setlength{\tabcolsep}{5pt}
\begin{tabular}{|l|c|c|c|c|}
\hline
 & \makecell{\textbf{Proposed with}\\\textbf{global basis}}
 & \makecell{\textbf{Proposed with}\\\textbf{local basis}}
 & \textbf{FERN}
 & \textbf{DeepONet} \\ \hline
\textbf{Relative error \%}
 & $0.08 \pm 0.01$
 & $0.71 \pm 0.01$
 & $1.13 \pm 0.03$
 & $0.09 \pm 0.02$ \\ \hline
\textbf{Total \# parameters}
 & $63,212$
 & $9376$
 & $21,824$
 & $75,660$ \\ \hline
\textbf{\# basis}
 & $32$
 & $32$
 & $32$
 & $32$ \\ \hline
\end{tabular}
\caption{Model comparison results for KPP example.}
\label{table_kpp_results}
\end{table}

\subsection{Aggregation-diffusion equation example}
\label{sec_aggregation}
In this section, we study the aggregation-diffusion equation,
\begin{align*}
    u_t - (u(D u^{m-1} + W(x)*u)_x)_x = 0, \ \ x\in[-6, 6], \ \ t\in[0, 200],
\end{align*} 
where $D = 0.4, m = 2$, $W(x) =  -\frac{1}{\sqrt{2 \pi}\sigma} \exp\left(-\frac{x^2}{2\sigma^2}\right)$ with $\sigma = 1$, and $*$ is the convolution.
The target operator is the mapping from the initial
condition to the solution at the terminal time. We generate the initial condition by sampling the free parameter $c_0$ uniformly in $[1, 5]$ from $$u_0(x) = \frac{c_0}{2\sqrt{2\pi}}(\exp(-(x-x_0)^2/2)+\exp(-(x+x_0)^2/2)).$$

We present the prediction demonstration in Figure \ref{fig_aggre_diff} and the numerical results in Table \ref{table_aggre_diff_example}. 
\begin{figure}[H]
    \centering
    \includegraphics[scale = 0.4]{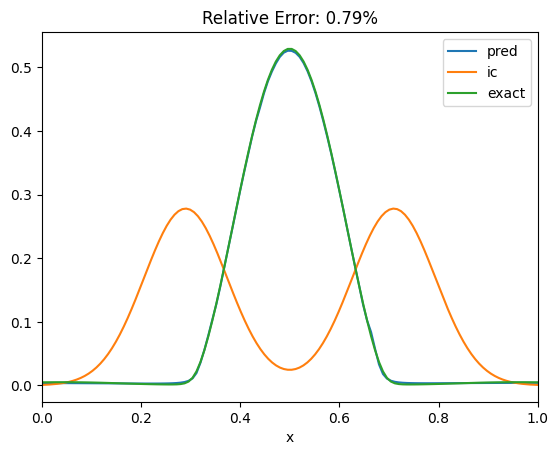}
    \includegraphics[scale = 0.4]{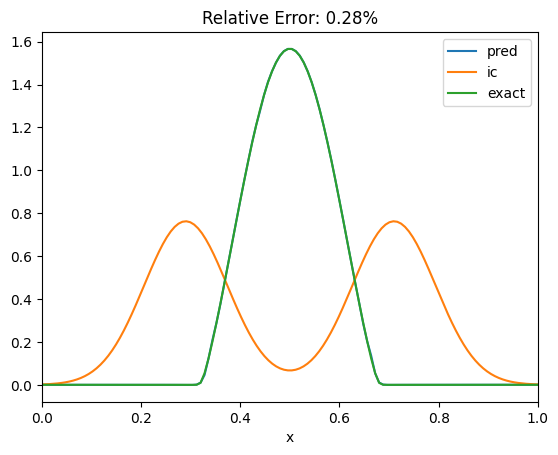}
    \caption{Two prediction demonstrations for the aggregation-diffusion equation example.}
    \label{fig_aggre_diff}
\end{figure}

\begin{table}[H]
\centering
\small
\setlength{\tabcolsep}{4pt}
\renewcommand{\arraystretch}{1.15}
\begin{tabular}{|l|c|c|c|c|}
\hline
& \shortstack{\textbf{Proposed with}\\\textbf{global basis}}
& \shortstack{\textbf{Proposed with}\\\textbf{local basis}}
& \textbf{FERN}
& \textbf{DeepOnet} \\
\hline
\textbf{Relative error (\%)}
& $0.28 \pm 0.17$
& $1.23 \pm 0.89$
& $0.94 \pm 1.18$ 
& $0.64 \pm 1.0$\\

\hline
\textbf{\# parameters}
& $63,212$
& $93,76$
& $21,824$ 
& $75,660$\\

\hline
\textbf{\# basis}
& $32$
& $32$
& $32$
& $32$ \\
\hline
\end{tabular}
\caption{Model comparison results for the aggregation-diffusion example.}
\label{table_aggre_diff_example}
\end{table}

\section{Conclusion}
In this paper, we discuss the connection between the numerical methods for solving PDE and the operator learning.
Specifically, we study one hypothesis: the trained neural operator for the forward PDe solving problem may approximate the numerical scheme which generates the training data.
Based on this hypothesis, we derive the neural operator from the numerical discretization.
The resulting network structure returns to the classical Chen-Chen neural operator form.
We also conduct the error analysis and derive the power scaling law which is widely observed in the experiments.


\newpage
\section{Appendix}
\label{sec_appendix}
\subsection{CFL.}
\label{section_cfl_appendix}
We derive the CFL condition for the scheme to make the manuscript complete.
The parabolic CFL restriction comes from the explicit treatment of diffusion.  To see the scaling, first consider the homogeneous heat equation, for which
$ u_h^{n+1}=\left(I-\dt M^{-1}S\right)u_h^n.$
We now derive the stability condition in the $M$-norm.  The relevant spectral problem is the generalized eigenvalue problem
$ S\psi_j=\lambda_jM\psi_j,
j=1,\ldots,N.$
Equivalently, $\psi_j$ is an eigenvector of $M^{-1}S$. 
$M^{-1}S$ is usually not symmetric in the Euclidean inner product, but it is self-adjoint in the $M$-inner product $ \langle a,b\rangle_M:=a^TMb.$
Indeed,
\begin{align*}
     \langle M^{-1}Sa,b\rangle_M=b^TSa=a^TSb=\langle a,M^{-1}Sb\rangle_M.
\end{align*}
We now claim $\{\psi_j\}$ forms a basis by a direct construction.
Because $M$ is symmetric positive definite, it has an orthogonal diagonalization
$M=U\diag(m_1,\ldots,m_N)U^T, m_i>0$, where $U$ is an orthonormal matrix.
We define
\begin{align*}
     M^{1/2}:=U\diag(m_1^{1/2},\ldots,m_N^{1/2})U^T,
 M^{-1/2}:=U\diag(m_1^{-1/2},\ldots,m_N^{-1/2})U^T.
\end{align*}
Thus $M^{1/2}M^{1/2}=M$, $M^{-1/2}M^{-1/2}=M^{-1}$, and $M^{-1/2}$ is invertible.  
Define matrix $ A:=M^{-1/2}SM^{-1/2}$, and $A$ is symmetric (since $M^{1/2}$ is symmetric and $S$ is also symmetric), so it has a Euclidean-orthonormal eigenbasis $z_1,\ldots,z_N$ with
$ A z_j=\lambda_j z_j.$
Set $\psi_j:=M^{-1/2}z_j$.  Multiplying $M^{-1/2}SM^{-1/2}z_j=\lambda_jz_j$ by $M^{1/2}$ gives
\begin{align*}
     S M^{-1/2}z_j=\lambda_j M^{1/2}z_j
 =\lambda_j M(M^{-1/2}z_j),
\end{align*}
hence $\psi_j$ satisfies $S\psi_j=\lambda_jM\psi_j$.
Moreover,
\begin{align*}
     \psi_i^TM\psi_j
 =z_i^TM^{-1/2}MM^{-1/2}z_j
 =z_i^Tz_j
 =\delta_{ij}.
\end{align*}
Because $M^{-1/2}$ is invertible, the vectors $\psi_1,\ldots,\psi_N$ are a basis of $\R^N$. 

The eigenvalues are characterized by the Rayleigh quotient.
Let, $w_h=\sum_j y_j\phi_j$
\begin{equation}\label{eqn_rayleigh}
 \lambda_{\max}(M^{-1}S)
 =
 \max_{0\ne y\in\R^N}
 \frac{y^TSy}{y^TMy}
 =
 \max_{0\ne w_h\in V_h}
 \frac{\norm{\nabla w_h}_{L^2(\Om)}^2}{\norm{w_h}_{L^2(\Om)}^2}.
\end{equation}
Now expand an arbitrary coefficient vector in this basis:
$ u_h^n=\sum_{j=1}^N c_j^n\psi_j.$
\medskip

The heat update maps each eigenmode independently.  Indeed, from $S\psi_j=\lambda_jM\psi_j$, multiplication by $M^{-1}$ gives $M^{-1}S\psi_j=\lambda_j\psi_j.$
Therefore we have,
\begin{align*}
     \left(I-\dt M^{-1}S\right)\psi_j
 =
 \psi_j-\dt\lambda_j\psi_j
 =
 (1-\dt\lambda_j)\psi_j.
\end{align*}
It follows that, $ u_h^{n+1}=\sum_{j=1}^N (1-\dt\lambda_j)c_j^n\psi_j.$
Using $M$-orthonormality,
\begin{align*}
     \norm{u_h^{n+1}}_M^2
 =
 \sum_{j=1}^N |1-\dt\lambda_j|^2 |c_j^n|^2,
 \quad
 \norm{u_h^n}_M^2
 =
 \sum_{j=1}^N |c_j^n|^2.
\end{align*}
Thus the heat update is non-expansive in the $M$-norm if and only if
$\max_{1\le j\le N}|1-\dt\lambda_j|\le 1.$
Since $\lambda_j\ge0$, this is equivalent to
$0\le \dt\lambda_j\le2, \text{for all }j.$
The most restrictive condition is therefore
\begin{align*}
     \dt\lambda_{\max}(M^{-1}S)\le 2.
\end{align*}

The finite element inverse inequality states that, for fixed polynomial degree $r$ on a shape-regular family of affine elements,
\begin{equation}\label{eqn_local_inverse}
 \norm{\nabla w_h}_{L^2(K)}
 \le C_{\rm inv} h_K^{-1}\norm{w_h}_{L^2(K)}
 \qquad\forall w_h\in V_h,
\end{equation}
where $C_{\rm inv}$ may depend on $r$, $d$, and the shape-regularity constant, but not on $h$.   Squaring and summing \eqref{eqn_local_inverse} over the elements,
and using $h_K\ge c_{\rm qu}h$ where $c_{\rm qu}$ is a constant, gives the global inverse estimate,
\begin{equation}\label{eq:global-inverse}
\|\nabla w_h\|_{L^2(\Omega)}
\le C_{\rm glob}h^{-1}\|w_h\|_{L^2(\Omega)},
\end{equation}
where $C_{\rm glob}=C_{\rm inv}/c_{\rm qu}$ is independent of $h$.
Combining \eqref{eqn_rayleigh} and \eqref{eq:global-inverse} yields
\begin{align*}
     \lambda_{\max}(M^{-1}S)\le C h^{-2}.
\end{align*}
For the fixed-degree Lagrange finite element spaces considered
here on shape-regular, quasi-uniform meshes, this scaling is
sharp, so $\lambda_{\max}(M^{-1}S)\sim h^{-2}$.

For the full semilinear equation, the explicit reaction term may impose an additional restriction, but this restriction is not a parabolic mesh restriction.  
Suppose the relevant exact and numerical values lie in a bounded
interval $[-M_*,M_*]$, with $M_*$ independent of $h$, and let
$L_g$ be a Lipschitz constant of $g$ on this interval.
Then the reaction contribution changes perturbations by at most a factor of size $1+\dt L_g$ per time step.  If one wants monotonicity or maximum-bound preservation, one may also need a condition of the form $\dt L_g\lesssim1$.  Crucially, $L_g$ depends on the bounded range but not on $h$.  For Fisher--KPP, for instance, $g(s)=s-s^2$, so $g'(s)=1-2s$ and one may take $L_g\le 1+2M_*$ on $[-M_*,M_*]$.  Thus the diffusion restriction $\dt\lesssim h^2$ is the stiff restriction as $h\to0$, and it controls the depth scaling of the time-unrolled network.

For the network complexity estimates, we choose an admissible
time step $\Delta t\sim h^2$ satisfying the diffusion CFL
condition. For fixed $T$, this gives
$N_t=T/\Delta t\sim h^{-2}$.
Without this additional choice, the restriction
$\Delta t\lesssim h^2$ only implies $N_t\gtrsim h^{-2}$.

\subsection{Proof of Lemma \ref{lemma_numerical_operator_bounds}}
\label{sec_stability_bnds}
In this section, we prove the stability estimates.
\begin{lemma*}[Basic stability bounds]
Assume the Assumption \ref{assumption_quadrature} hols.
Then we have $ \norm{B}_{M\to W}=1$, $\norm{C_h}_{W\to M}\le \dt$, $ \sqrt{\Tr(W)}=|\Om|^{1/2}.$
Finally, if the explicit diffusion CFL condition $\dt\lambda_{\max}(M^{-1}S)\le2$ holds, then $ \norm{L_h}_{M\to M}\le1+\dt.$
\end{lemma*}

\begin{proof}
Mass exactness followed by Assumption \ref{assumption_quadrature} gives
\begin{align*}
     \norm{Be}_W^2=e^TB^TWBe=e^TMe=\norm{e}_M^2.
\end{align*}
Therefore, by the definition \eqref{eqn_operator_norms_1}, we have $\norm{B}_{M\to W}=1$.
For $C_h = -\dt M^{-1}B^TW\in \mathbb{R}^{N\times Q}$, let $z\in\R^Q$ be arbitrary, and set $ x:=M^{-1}B^TWz.$
We estimate $\norm{x}_M$.  The dual representation of the norm induced by the inner product $\langle a,b\rangle_M=a^TMb$ is
\begin{equation}\label{eq:M-duality}
 \norm{x}_M
 =
 \sup_{y\ne0}\frac{\abs{\langle x,y\rangle_M}}{\norm{y}_M}.
\end{equation}
Since $Mx=B^TWz$ and the $M$ adjointness of the inner product $\langle \cdot, \cdot\rangle$ (See Appendix \ref{section_cfl_appendix}) for the details,
\begin{align*}
     \langle x,y\rangle_M
 =
 x^TMy
 =
 y^TMx
 =
 y^TB^TWz
 =
 (By)^TWz.
\end{align*}
By the Cauchy--Schwarz inequality in the $W$-inner product and by $\norm{B}_{M\to W}=1$,
\begin{align*}
     \abs{(By)^TWz}
 \le
 \norm{By}_W\norm{z}_W
 \le
 \norm{y}_M\norm{z}_W.
\end{align*}
Inserting this into \eqref{eq:M-duality} gives
\begin{align*}
     \norm{M^{-1}B^TWz}_M\le \norm{z}_W.
\end{align*}
Multiplying by $\dt$ proves $\norm{C_h}_{W\to M}\le \dt$.
Next, we prove $\sqrt{\Tr(W)}=|\Om|^{1/2}$.
Since $W=\diag(\omega_1,\dots,\omega_Q)$, $\Tr(W)=\sum_{q=1}^Q\omega_q$.
If the quadrature is exact for constants, then $\sum_q\omega_q=\int_\Om1\,dx=|\Om|$.  
In the degree-$3r$ exact setting of Section~\ref{subsec_kpp-quad}, exactness for constants is automatic because constants are polynomials of degree zero.

\medskip
It remains to prove $L_h$ operator norm.  
Because $M$ is symmetric positive definite and $S$ is symmetric, the generalized eigenproblem
\begin{align*}
     S\psi_j=\lambda_jM\psi_j,
 \qquad j=1,\ldots,N,
\end{align*}
has $N$ generalized eigenvectors that may be chosen $M$-orthonormal and therefore forms an $M$-orthonormal basis of $\R^N$. 
Therefore every $e\in\R^N$ has a unique expansion $e=\sum_{j=1}^N c_j\psi_j$.
By the definition of $L_h$, we have
\begin{align*}
     L_h\psi_j
 =
 \left(I+\dt I-\dt M^{-1}S\right)\psi_j
 =
 (1+\dt-\dt\lambda_j)\psi_j
 =
 \mu_j\psi_j.
\end{align*}
Under $\dt\lambda_{\max}\le2$ CFL condition, we have $\dt\lambda_j\in[0,2]$, hence
\begin{align*}
     \abs{\mu_j}=\abs{1+\dt-\dt\lambda_j}\le1+\dt.
\end{align*}
Using the expansion above,
\begin{align*}
     L_he=\sum_{j=1}^N c_jL_h\psi_j
      =\sum_{j=1}^N c_j\mu_j\psi_j.
\end{align*}
By $M$-orthonormality of $\psi_j$, all cross terms vanish:
\begin{align*}
     \norm{e}_M^2
 =\left(\sum_i c_i\psi_i\right)^TM\left(\sum_j c_j\psi_j\right)
 =\sum_{i,j}c_ic_j\psi_i^TM\psi_j
 =\sum_{j=1}^N\abs{c_j}^2,
\end{align*}
and similarly
\begin{align*}
     \norm{L_he}_M^2
 =\left(\sum_i c_i\mu_i\psi_i\right)^TM
  \left(\sum_j c_j\mu_j\psi_j\right)
 =\sum_{j=1}^N\abs{\mu_j}^2\abs{c_j}^2.
\end{align*}
Thus
\begin{align*}
     \norm{L_he}_M
 \le
 \max_{1\le j\le N}\abs{\mu_j}\,\norm{e}_M
 \le
 (1+\dt)\norm{e}_M.
\end{align*}
Taking the supremum over $e\ne0$ gives $\norm{L_h}_{M\to M}\le1+\dt$.
\end{proof}

\subsection{Proof of the approximating to FEM basis}
\label{sec_appendix_proof_of_fem_basis_approx}
In this section we provide the proof to Lemma \ref{lemma_basis_network_nonlocal}.

\begin{proof}
Fix a basis index $j$ and $0<\gamma<1$. Throughout the proof, constants may depend on the fixed degree $r$, dimension $d$, reference nodal basis, and the shape-regularity and quasi-uniformity constants, but are independent of $h$, $j$, and $\gamma$.

Extend $\phi_j$ by zero outside $\Omega$, and denote its extension by $\bar\phi_j$. Since $\phi_j$ is continuous and has zero boundary trace, its zero extension is continuous and belongs to $W^{1,\infty}(\mathbb R^d)$. Standard scaling estimates for fixed-degree nodal Lagrange basis functions give
$\|\bar\phi_j\|_{L^\infty(\mathbb R^d)}
+h\|\nabla\bar\phi_j\|_{L^\infty(\mathbb R^d)}\le C_0$.
Moreover, the support of $\bar\phi_j$ has diameter at most $Ch$. We may therefore choose a closed box
$D_{j,h}=\prod_{i=1}^d[a_{j,i},a_{j,i}+\ell_{j,i}]$
such that
$\operatorname{supp}\bar\phi_j\subset\operatorname{int}D_{j,h}$
and $ch\le\ell_{j,i}\le Ch$ for every $i$. The box need not lie inside $\Omega$.

Define the affine bijection
$F_{j,h}:\mathbb R^d\to\mathbb R^d$
by
$(F_{j,h}(y))_i=a_{j,i}+\ell_{j,i}y_i$.
Then $F_{j,h}([0,1]^d)=D_{j,h}$, and its inverse is defined for every $x\in\mathbb R^d$ by
$(F_{j,h}^{-1}(x))_i=(x_i-a_{j,i})/\ell_{j,i}$.
On the reference cube, set
$\psi_{j,h}=\bar\phi_j\circ F_{j,h}$.
We have
$\|\psi_{j,h}\|_{L^\infty([0,1]^d)}\le C_0$,
and the weak chain rule gives
$\partial_{y_i}\psi_{j,h}
=\ell_{j,i}(\partial_{x_i}\bar\phi_j)\circ F_{j,h}$
almost everywhere. Consequently,
$\|\partial_{y_i}\psi_{j,h}\|_{L^\infty([0,1]^d)}
\le \ell_{j,i}C_0h^{-1}\le C$.
Thus there exists a constant $A\ge1$, independent of $h$ and $j$, such that
$\|\psi_{j,h}\|_{W^{1,\infty}([0,1]^d)}\le A$.
Our choice of the box also ensures that
$\psi_{j,h}=0$ on $\partial[0,1]^d$.

Apply Yarotsky's approximation theorem' with $n=1$, to
$\psi_{j,h}/A$ with accuracy $\gamma/A$, and multiply the network output by $A$.
This gives a ReLU network $\widehat\Phi_{j,h,\gamma}$ satisfying
$\sup_{y\in[0,1]^d}
|\widehat\Phi_{j,h,\gamma}(y)-\psi_{j,h}(y)|\le\gamma$.
The estimate holds on the closed cube because both functions are continuous.
Since $A$ is independent of $h$ and $j$, the construction has depth at most
$C(1+\log(1/\gamma))$, width at most $C\gamma^{-d}$, and size at most
$C\gamma^{-d}(1+\log(1/\gamma))$.
Here size denotes the number of nonzero parameters.
The width bound follows from the parallel realization of
$O((\gamma/A)^{-d})$ local subnetworks in the theorem's construction,
each with width bounded in terms of $d$.

Let $\sigma(t)=\max\{t,0\}$, and define the coordinatewise clipping map
$\Pi:\mathbb R^d\to[0,1]^d$ by
$\Pi(y)_i=\sigma(y_i)-\sigma(y_i-1)$.
For every $x\in\mathbb R^d$, set
$\Phi_{j,h,\gamma}(x)
=\widehat\Phi_{j,h,\gamma}\bigl(\Pi(F_{j,h}^{-1}(x))\bigr)$.

If $x\in D_{j,h}$, then $F_{j,h}^{-1}(x)\in[0,1]^d$, so clipping leaves this point unchanged. Hence
$|\Phi_{j,h,\gamma}(x)-\bar\phi_j(x)|
=|\widehat\Phi_{j,h,\gamma}(F_{j,h}^{-1}(x))
-\psi_{j,h}(F_{j,h}^{-1}(x))|
\le\gamma$.

If $x\notin D_{j,h}$, at least one coordinate of
$F_{j,h}^{-1}(x)$ lies outside $[0,1]$. Therefore,
$z:=\Pi(F_{j,h}^{-1}(x))\in\partial[0,1]^d$,
where $\psi_{j,h}(z)=0$.
Also, $\bar\phi_j(x)=0$, since its support lies inside $D_{j,h}$.
It follows that
$|\Phi_{j,h,\gamma}(x)-\bar\phi_j(x)|
=|\widehat\Phi_{j,h,\gamma}(z)-\psi_{j,h}(z)|
\le\gamma$.
Combining the two cases yields
$\|\Phi_{j,h,\gamma}-\bar\phi_j\|_{L^\infty(\mathbb R^d)}\le\gamma$,
and consequently
$\|\Phi_{j,h,\gamma}-\phi_j\|_{L^\infty(\Omega)}\le\gamma$.

Finally, the affine map $F_{j,h}^{-1}$ and the clipping map are represented exactly by networks with complexity bounded in terms of $d$. Composing these with $\widehat\Phi_{j,h,\gamma}$ preserves the stated depth, width, and size bounds up to uniform constants. Although the affine weights may grow as $h^{-1}$, their number is independent of $h$. This proves the lemma.
\end{proof}

\subsection{Proof of the approximation with local support}
\label{sec_proof_of_local_approx_lemma}
In this section, we prove the Lemma \ref{lemma_support_preserving_approximation} ReLU approximation where the network has a local support.

\begin{proof}
Because $n\geq1$, the function $f$ is Lipschitz, or 
$\|f(x) - f(y)\|\leq L_f \|x-y\|_{\infty}$ for $x, y$ in the domain, where $L_f$ is the Lipschitz constant.
Besides, let $ B_f:=\norm{f}_{L^\infty(\mathbb R^d)}+1.$
Since $f$ is continuous and vanishes outside the closed box $D$, it
vanishes on $\partial D$. 
Hence
\begin{equation}\label{eqn_support_boundary_distance_local}
 \abs{f(x)}
 \leq
 L_f\operatorname{dist}(x,\partial D),
 \quad x\in D.
\end{equation}

Choose a buffer width $\kappa$ satisfying $ 0<\kappa<\frac12\min_{1\leq i\leq d}(b_i-a_i),$
and define $ D_\kappa
 :=
 \prod_{i=1}^d[a_i+\kappa,b_i-\kappa].$
For each coordinate, define the exact ReLU trapezoid
\begin{equation}\label{eq:coordinate-mask-local}
 \psi_{i,\kappa}(t)
 :=
 \frac1\kappa
 \Bigl[
 \relu(t-a_i)
 -\relu(t-(a_i+\kappa))
 -\relu(t-(b_i-\kappa))
 +\relu(t-b_i)
 \Bigr].
\end{equation}
Then $0\leq\psi_{i,\kappa}\leq1$, the function
$\psi_{i,\kappa}$ vanishes outside $[a_i,b_i]$, and it equals one on
$[a_i+\kappa,b_i-\kappa]$; 
it continuous linearly connect between $a_i$ and $a_i + \kappa$ and $b_i-\kappa$ and $b_i$.
Now we use the minimum function constructed by constant size ReLU network,
\begin{align*}
     \min\{s,t\}
 =
 \frac12\bigl(s+t-\relu(s-t)-\relu(t-s)\bigr),
\end{align*}
and construct a mask,
\begin{align*}
     M_\kappa(x)
 :=
 \min_{1\leq i\leq d}\psi_{i,\kappa}(x_i), \quad x = [x_1, ..., x_d]^{\intercal}.
\end{align*}
Thus $ 0\leq M_\kappa\leq1,
 M_\kappa=1\ \text{on }D_\kappa,
 M_\kappa=0\ \text{on }\mathbb R^d\setminus D.$
Since mask is implemented by the ReLU structure of minimizing $d$ inputs, we can use the binary tree structure. Consequently, the mask has depth $\sim 1+\log d$ and width $\sim d$.

Let $0<\rho,\mu<1$ be determined later.  
The target domain $D$ is a cube,
then apply an affine rescaling of $D$ to the unit cube and
normalization by the $W^{n,\infty}$-norm,
Yarotsky's approximation Lemma \ref{lemma_yarotsky} gives a ReLU network $F_\rho$ such that
$ \norm{f-F_\rho}_{L^\infty(D)}\leq\rho.$
Its depth is $O(1+\log(1/\rho))$, and its size is
$O(\rho^{-d/n}(1+\log(1/\rho)))$ with width $\rho^{-d/n}$.  
Besides, $\abs{F_\rho(x)}\leq B_f,$ for $x\in D$.
\medskip

Let $\operatorname{Mult}_{\mu,B_f}$ be a ReLU multiplication network in Lemma \ref{lemma_yarotsky_product} such that
\begin{align}
 \abs{\operatorname{Mult}_{\mu,B_f}(s,t)-st}
 &\leq\mu, \text{ for }\abs{s}\leq B_f,\ \abs{t}\leq B_f, \nonumber
 \\
 \operatorname{ Mult}_{\mu,B_f}(s,0)
 &=0, \text{ for every }s\in\mathbb R.
 \label{eqn_multiplication_properties}
\end{align}
This network has width order constant, depth and size $\sim 1+\log(1/\mu)$.  
Define
\begin{align*}
     \widetilde f_\varepsilon(x)
 :=
 \operatorname{Mult}_{\mu,B_f}\bigl(F_\rho(x),M_\kappa(x)\bigr).
\end{align*}
We now study the value of $\tilde{f}_{\varepsilon}$ at different intervals.

(1) If $x\notin D$, then $M_\kappa(x)=0$, and
\eqref{eqn_multiplication_properties} gives
$\widetilde f_\varepsilon(x)=0$.  
Hence
$\operatorname{supp}(\widetilde f_\varepsilon)\subseteq D$.

(2) If $x\in D_\kappa$, then $M_\kappa(x)=1$. 
By the
multiplication approximation property,
$|\widetilde f_\varepsilon(x)-f(x)|\leq\mu+\rho$.
(3) If $x\in D\setminus D_\kappa$, then
$\operatorname{dist}(x,\partial D)\leq\kappa$.  
Using
\eqref{eqn_support_boundary_distance_local}, $0\leq M_\kappa\leq1$, and
\eqref{eqn_multiplication_properties}, we obtain
\begin{align*}
 \abs{\widetilde f_\varepsilon(x)-f(x)}
 &\leq
 \abs{\operatorname{Mult}_{\mu,B_f}(F_\rho,M_\kappa)-F_\rho M_\kappa}
 +\abs{(F_\rho-f)M_\kappa}
 +\abs{f(M_\kappa-1)}\\
 &\leq
 \mu+\rho+L_f\kappa.
\end{align*}
Note that the first estimate above is true as $B_f >1$ by the definition. 
Choose
$$
 \rho=\frac\varepsilon3,
 \quad
 \mu=\frac\varepsilon3,
 \quad
\kappa
=
\min\left\{
\frac14\min_{1\le i\le d}(b_i-a_i),
\frac{\varepsilon}{3\max\{L_f,1\}}
\right\}.
$$
Since $\mu\sim\rho\sim\varepsilon$, the size of $F_{\rho}$ dominates the entire network size over the mask $M_{\kappa}$ and $\operatorname{Mult}_{\mu,B_f}$.
Consequently, we have the estimate as in the theorem argument.

\end{proof}

\bibliographystyle{plain}
\bibliography{references}

\end{document}